\documentclass[11pt, oneside, letterpaper, reqno]{article}

\usepackage{amsthm, amssymb, amscd, amsmath}
\usepackage{graphicx, color}
\usepackage{mathrsfs}

\usepackage{enumerate}

\usepackage[hidelinks, colorlinks=true, linkcolor = blue, linktoc = page, citecolor = blue]{hyperref}

\usepackage[T1]{fontenc}
\usepackage{lmodern}

\usepackage[text={5.8in, 8.0in},centering]{geometry}				
\usepackage{microtype}
\numberwithin{equation}{section}

\newtheorem{theorem}{Theorem}[section]  
\newtheorem*{theorem*}{Theorem}
\newtheorem{proposition}[theorem]{Proposition}

\newtheorem{corollary}[theorem]{Corollary}
\newtheorem{lemma}[theorem]{Lemma}

\theoremstyle{definition}
\newtheorem{remark}{Remark}

\newcommand{\C}{\mathbb{C}}
\newcommand{\Q}{\mathbb{Q}}
\DeclareMathOperator{\Span}{\mathrm{span}}
\newcommand{\st}{:\,}
\newcommand{\R}{\mathbb{R}}
\newcommand{\cE}{\mathcal{E}}
\renewcommand{\ker}{\mathrm{ker}\,}
\newcommand{\cL}{\mathcal{L}}
\newcommand{\Z}{\mathbb{Z}}
\newcommand{\N}{\mathbb{N}}

\newcommand{\qu}[1]{q_{[#1]}}

\newcommand{\phiu}[1]{\varphi_{[#1]}}
\newcommand{\rhou}[1]{\rho_{[#1]}}

\newcommand{\epsu}[1]{\epsilon_{[#1]}}
\newcommand{\id}{\mathrm{id}}
\newcommand{\qn}{q^*}
\newcommand{\phin}{\varphi^*}
\newcommand{\rhon}{\rho^*}

\newcommand{\bv}{\mathbf{v}}
\newcommand{\bw}{\mathbf{w}}

\newcommand{\floor}[1]{\lfloor #1 \rfloor}
\newcommand{\diag}{\mathrm{diag}\,}
\newcommand{\bmat}[1]{\begin{bmatrix}#1\end{bmatrix}}

\newcommand{\const}{\mathrm{const}}

\title{Gevrey regularity for the formally linearizable billiard of Treschev}

\begin{document}

\author{Qun Wang and Ke Zhang}

\maketitle

\begin{abstract}
  Treschev  made the remarkable discovery that there exists formal power series describing a billiard with locally linearizable dynamics. We show that if the frequency for the linear dynamics is Diophanine, the Treschev example is $(1+ \alpha)$-Gevrey for some $\alpha > 0$. Our proof is based on an iterative scheme that further clarifies the structure and symmetries underlying the original Treschev construction. Hopefully, Our result sheds a light on the more important question of whether this example is convergent.
\end{abstract}

\section{Introduction}
The Birkhoff-Poritsky conjecture (\cite{Por50}) states that the only integrable billiards are the ellipses. Several advances are made recently towards this conjecture, see \cite{ADSK16, KS18, BM22}. In these works, it is assumed that a certain part of the phase space is foliated by \emph{essential} invariant curves. It is an open question whether the analogous conjecture is true if the billiard is integrable near a periodic orbit. In this case, an open set of the phase space is foliated by \emph{contractible} invariant curves. 

Treschev (\cite{Tre13, Tre15, Tre17}) discovered a billiard whose dynamics near a period two orbit is formally linearizable. (See also \cite{Tre22} for an analogous result for Hamiltonian systems). If this example were to converge, then the local version of Birkhoff-Poritsky conjecture is false. Treschev's example is a billiard played between two mirror symmetric curved walls, with each wall symmetric with respect to the horizontal axis, see Figure \ref{fig:treschev}. To give a precise description, we use an alternative coordinate system due to Bialy and Mironov (\cite{BM22}).

\begin{figure}[ht]
 \centering 
 \includegraphics[width=3in]{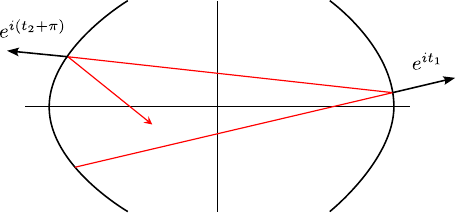} 
 \caption{Treschev's example}\label{fig:treschev}
\end{figure}

Let $\Omega$ be a strictly convex subset of $\R^2$. Define its support function by
\[
  q(\psi) = \sup \{ \langle z, e^{i\psi} \rangle \st z \in \Omega \},\quad 
  \psi \in \R/(2\pi \Z),
\]
where we did and will continue to identify $\C$ and $\R^2$ in the notations. The function
\[
  q(\psi) e^{i\psi} + i q'(\psi) e^{i\psi}
\]
provides a parametrization of $\partial \Omega$. We assume that $q(\psi)$ is equal to a real analytic function near $\psi = 0$, and admits a $(\Z_2 \times \Z_2)$-symmetry, namely
\[
q(\psi) = q(-\psi), \quad q(\pi - \psi) = q(\psi).
\]
In particular, $q$ also admits the central symmetry $q(\psi + \pi) = q(\psi)$. 

Represent a sequence of billiard trajectories by its tangent angles $(t_n)$. In these coordinates, $(0, \pi, 0, \pi, \cdots)$ corresponds to the horizontal two-periodic orbit. By properly choosing the curvature of the boundaries, one can make the two-periodic orbit elliptic, hence KAM stable. The nearby orbits alternate between $t \in (-\frac{\pi}{2}, \frac{\pi}{2})$ and $t \in (\frac{\pi}{2}, \frac{3\pi}{2})$. Due to the central symmetry of the boundary, we can identify $t$ and $t + \pi$ and represent a billiard orbit by the sequence $t_n \in (-\frac{\pi}{2}, \frac{\pi}{2})$, $n \in \Z$. Using the alternative generating function discovered by Bialy and Mironov (\cite{BM22}), we will show that  (See Lemma \ref{lem:t-var}) $(t_1, t_2, t_3)$ represents a billiard trajectory if and only if 
\begin{equation}  \label{eq:intro-billiard-eq}
  \partial_2 S(t_1, t_2) + \partial_1 S(t_2, t_3) = 0, \quad
  S(t_1, t_2) = S_q(t_1, t_2) = q\left( \frac{t_1 + t_2}{2}\right) \cos\left( \frac{t_1 - t_2}{2}\right), 
\end{equation}
where $q$ is the support function of the boundary.
We then define the billiard map $T: (t_1, t_2) \mapsto (t_2, t_3)$ by \eqref{eq:intro-billiard-eq}. By reducing the central symmetry, the horizontal orbit becomes a fixed point, i.e. $T(0, 0) = (0, 0)$.

Suppose $(0, 0)$ is a lineariazble elliptic fixed point of $T$. Then there exists a rotation $R_\theta$ and a change of variable $\Phi: \R^2 \to \R^2$ such that 
\begin{equation} \label{eq:conjugacy}
  T \circ \Phi = \Phi \circ R_\theta,
\end{equation}
where $ 
  R_\theta = \begin{pmatrix}
  \cos{\theta} & - \sin{\theta} \\
  \sin{\theta} & \cos{\theta}
  \end{pmatrix}$. Following Treschev (\cite{Tre13}), we write $\Phi = (\varphi_1, \varphi_2)$ as a function of $z, \bar{z}$, where $z, \bar{z} \in \C$. This way, $R_\theta(z,\bar{z}) = (\lambda z, \lambda^{-1} \bar{z})$ and the conjugacy equation \eqref{eq:conjugacy} can be conveniently written as 
\[
  T(\varphi_1(z,\bar{z}), \varphi_2(z,\bar{z})) 
  = \left(\phi_1\circ R_\theta(z,\bar{z}), \phi_2\circ R_\theta(z,\bar{z}) \right)
  =\left(\varphi_1(\lambda z,\lambda^{-1}\bar{z}), \varphi_2(\lambda z,\lambda^{-1}\bar{z})\right), 
\]
where $\lambda = e^{i\theta} \in \C$. Since $T(t_1, t_2) = (t_2, t_3)$, we get $\varphi_2(z,\bar{z}) = \varphi_1(\lambda z,\lambda^{-1}\bar{z})$. Plug into \eqref{eq:intro-billiard-eq}, we get 
\[
  \partial_2 S(\varphi_2(\lambda^{-1} z, \lambda \bar{z}), \varphi_2(z, \bar{z})) + \partial_1 S(\varphi_2(z, \bar{z}), \varphi_2(\lambda z, \lambda^{-1}\bar{z})) = 0. 
\]
Denote $\varphi = \varphi_2$, and write $\varphi(\lambda^{-1} z,\lambda \bar{z}) = \varphi^-(z,\bar{z})$, $\varphi(\lambda z,\lambda^{-1}\bar{z}) = \varphi^+(z,\bar{z})$, we get the following equation
\begin{equation}  \label{eq:linearizable}
  \cE(q, \varphi) := \partial_2 S(\varphi^-, \varphi) + \partial_1 S(\varphi, \varphi^+) = 0. 
\end{equation}
Henceforth
\begin{itemize}
    \item  we assume that $\displaystyle \varphi(z,\bar{z})=\sum_{ j,k=0}^{\infty}\varphi_{jk}(z^{j}\bar{z}^{k})$ is a complex formal series in $z,\bar{z}$. Then $\varphi$ represents a real function if and only if $\overline{\varphi_{jk}}=\varphi_{kj}$. In the sequel the solution that we investigate will satisfy the stronger condition $\varphi_{jk}=\varphi_{kj}$.
    \item We use the notation $O_{k}$ to represent $\displaystyle O(\sum_{\substack{\alpha\geq0, \beta \geq 0\\ \alpha+\beta=k}}|z|^\alpha |\bar{z}|^{\beta})$
\end{itemize}
 
The following theorem reformulates the result of Treschev.
\begin{theorem}[see also \cite{Tre13}]
  For each $\lambda = e^{i\theta}$ where $\theta \in \R \setminus \Q$, there exists real formal power series 
  \[
    q(t) = \sum_{k = 0}^\infty q_{2k} t^{2k}, \quad 
    \varphi(z, \bar{z}) = \sum_{n = 0}^\infty \sum_{j + k = 2n + 1} \varphi_{j, k} z^j \bar{z}^k
  \]
  where $\varphi_{j, k} = \varphi_{k, j}$, 
  solving equation \eqref{eq:linearizable} as formal power series in $z, \bar{z}$. 
\end{theorem}

Assume the following Diophantine property for $\lambda$: there exists $1>c> 0$ and $\tau > 0$ such that 
\begin{equation}  \label{eq:diophantine}
  |\lambda^k - 1| \ge c |k|^{-\tau}, \quad
  \text{ for all } k \ne 0. 
\end{equation}
Our main theorem is that this formal solution is of Gevrey class.

\begin{theorem}\label{thm:main}
  For every $\alpha > \frac{5}{4}$, there exists $C(c,\tau) > 0$ such that
\[
  |q_{2k}| \le C^{2k} e^{\alpha (2k) \log (2k)} , \quad 
  \sup_{j + k = n} |\varphi_{j, k}| \le C^n e^{ \alpha n \log n}. 
\]
In particular, the series are of Gevrey order $1 + \alpha$. 
\end{theorem}
\begin{remark}
The Gevrey order we obtained is independent of the parameters of the Diophantine condition, because the features of the proof that cause faster growth of coefficients have a much larger effect than the small denominators coming from the Diophantine condition. 

In fact, we expect the proof can be adapted to work under the condition
\[
  |\lambda^k - 1| > c e^{-\gamma |k|}, \quad k \ne 0
\]
with a small $\gamma$. In this case we expect the Gevrey order to depend on $\gamma$. This is not done in this paper, to avoid excessive technicality. 
\end{remark}

Our proof proceeds by a KAM-type iterative scheme. Given an initial guess $(q, \varphi)$, write
\[
   \cE(q + \Delta q, \varphi + \Delta \varphi) = \cE(q, \varphi) + \partial_q \cE(q, \varphi) (\Delta q) + \partial_\varphi \cE(q, \varphi) (\Delta \varphi) + O_2
\]
where $O_2$ is a higher order remainder. Using a Newton scheme, we need to solve the linearized equation
\[
  \partial_q \cE(q, \varphi) (\Delta q) + \partial_\varphi \cE(q, \varphi) (\Delta \varphi) 
= - \cE(q, \varphi).  
\]
Two observations allow us to simplify this equation. First of all, the functional $\cE$ is \emph{linear} in $q$, therefore 
\[
\cE(q, \varphi) +  \partial_q \cE(q, \varphi)\Delta q =\cE(q, \varphi)+ \cE(\Delta q, \varphi) = \cE(q+\Delta q, \varphi). 
\]
Secondly, using the Lagrangian setting of Levi and Moser (\cite{LeviMoserLagrangian2001}), we have
\[
  \partial_\varphi \cE(q, \varphi) (\Delta \varphi) \cdot \varphi_z = \cL_z (\Delta \varphi) + O_2,  
\]
where $\cL_z$ is a second order difference operator to be defined later. Multiplying by $\varphi_z$ and ignoring the higher order term, the linearized equation can be rewritten as
\begin{equation}  \label{eq:intro-linear-eq}
  \cL_z (\Delta \varphi) = - \cE(q + \Delta q, \varphi) \varphi_z.
\end{equation}
Roughly speaking, we solve \eqref{eq:intro-linear-eq} in two steps:
\begin{enumerate}[(1)]
  \item Define an operator $(q, \varphi) \mapsto \Delta q$, so that $\cE(q + \Delta q, \varphi)$ projects to the image of $\cL_z$.
  \item Find an approximate inverse of $\cL_z$ on its image.
\end{enumerate}

It turns out that the inverse operator in step (2) is \emph{tame} in sense of KAM theory. However, the operator in step (1) is \emph{unbounded} in the analytic norm. This unboundedness makes the KAM scheme diverge on the space of analytic functions. Yet by modifying the KAM scheme and allowing the domain of analyticity to shrink to $0$, we obtain a Gevrey estimate. The same idea is used in \cite{BL22} to prove Gevrey estimate in the context of KAM theory for conformally symplectic systems.

Here is a basic outline of the paper:
\begin{itemize}
    \item In section 2, after adapting the generating function of Bialy and Mironov \cite{BM22} into a version suitable to our setting, we apply the method of Levi-Moser \cite{LeviMoserLagrangian2001} to obtain the cohomological equation \eqref{eq:app-lin}, and decompose it into an \textit{outer part} as well as an \textit{inner part}, for further investigation. 
    \item In section 3, we solve the outer cohomological equation formally by choosing $\Delta q$ by forcing the averaged generating function to vanish, which in turn implies the invertibility of the $\nabla^{+}$ operator. 
    \item In section 4, we solve the inner cohomological equation approximately, with a prescribed order of truncation, followed by a symmetrisation of the output. Depending on whether the source of error is related to the average of generating function or not, we decompose the error into two parts for further analysis. The process of solving and estimating $\Delta q$ is technical and is presented in detail in the appendix. 
    \item In section 5, after claiming the iterative scheme, we exhibit a series of properties of our chosen norm that will facilitate our estimate of magnitude of solutions produced in each iteration. As a byproduct, we obtain an alternative proof for the existence of solution as formal series, i.e. Treschev's original result (\cite{Tre13}).
    \item In section 6, after choosing an appropriate initial point, we launch the KAM machine with all the previous preparation, which leads to the Gevrey regularity we seek.  
\end{itemize}

\section{Basic calculations and the linearized equation}

Let us first derive the generating function based on the coordinate $t$.
\begin{proposition}[\cite{BM22}]
  Let $q: \R/(2\pi \Z) \to \R$ be the support function of a strictly convex billiard domain $\Omega$. Represent the direction of each billiard trajectory by its outward normal vector $e^{i\phi}$, then $\phi_1, \phi_2, \phi_3$ represent consecutive rays of the billiard trajectories if and only if 
\begin{equation}  \label{eq:Lag-eq}
  \partial_2 S(\phi_1, \phi_2) + \partial_1 S(\phi_2, \phi_3) = 0, 
\end{equation}
where
\[
  S(\phi_1, \phi_2) = q\left( \frac{\phi_1 + \phi_2}{2}\right) \sin \left( \frac{\phi_2 - \phi_1}{2}\right), 
\]
where the variable $\phi_2 - \phi_1$ should be considered to be in $[0, 2\pi)$.
(Our choice differs from \cite{BM22} by a factor of $2$, but this does not affect the equation).  
\end{proposition}

Let $(\phi_n)$ represent the rightward normal vector of a billiard trajectory that bounces between the left and right boundary, with $\phi_0 \in (-\pi, 0)$. Then there exists unique $t_n \in (-\frac{\pi}{2}, \frac{\pi}{2})$ such that
\begin{equation}  \label{eq:phi-to-t}
  (\phi_n)_{n \in \Z} = (t_n +  \frac{\pi}{2} + n\pi)_{n \in \Z}. 
\end{equation}
Note that 
\begin{align*}
    q\left(\frac{\phi_1+ \phi_2}{2}\right)= q\left(\frac{t_1+t_2}{2}\right), \sin \left(\frac{\phi_2-\phi_1}{2}\right) = \cos \left( \frac{t_1-t_2}{2}\right)
\end{align*}
Pluging into the equation \eqref{eq:Lag-eq}, we get:
\begin{lemma}\label{lem:t-var}
  Let $(t_1, t_2, t_3)$	be related to $(\phi_1, \phi_2, \phi_3)$ by \eqref{eq:phi-to-t}. Then $(t_1, t_2)$ is mapped to $(t_2, t_3)$ by the billiard map if and only if
\begin{equation}  \label{eq:billiard-eq}
  \partial_2 S(t_1, t_2) + \partial_1 S(t_2, t_3) = 0, \quad
  S(t_1, t_2) = q\left( \frac{t_1 + t_2}{2}\right) \cos\left( \frac{t_1 - t_2}{2}\right). 
\end{equation}
\end{lemma}

We will normalize $q_0 = 1$, then write $q(t) = 1 + q_2 t^2 + q_{4+}(t)$ and  $\cos(t) =: p(t) = 1 + p_2 t^2 + p_{4+}(t)$. Then
\[
  S(t_1, t_2)  = 1 + \frac14 q_2 (t_1 + t_2)^2 + \frac14 p_2(t_1 - t_2)^2 + O_4(t_1, t_2)
\]
and equation \eqref{eq:linearizable} becomes 
\[
  \cE(q, \varphi) = 
  \frac12( (q_2 - p_2) \varphi^- + 2 (q_2 + p_2) \varphi + (q_2 - p_2) \varphi^+)
  + O_3(\varphi).
\]
After normalization, we can always assume that $\varphi(z, \bar{z}) = z + \bar{z} + O_3$. Set 
\[
  \chi(x) = (q_2 - p_2) x^{-1} + 2(q_2 + p_2) + (q_2 - p_2) x, 
\]
we check that 
\[
  \cE(q, \varphi) = \frac12 \chi(\lambda^{-1}) z + \frac12 \chi(\lambda) \bar{z} + O_3.  
\]
We now choose $q_2$ such that $\chi(\lambda) = \chi(\lambda^{-1}) = 0$,  or
\begin{equation}  \label{eq:q2}
  q_2 = p_2 \frac{(\lambda - 1)^2}{(\lambda + 1)^2}.
\end{equation}
Set $\phiu{0}= z + \bar{z}$, $\qu{0}= 1 + q_2 t^2$, we get
\[
  \cE(\qu{0}, \phiu{0}) = O_3. 
\]

As mentioned in the introduction, we will device a KAM-type iterative scheme by solving the linearized equation
\begin{equation}  \label{eq:lin-eq}
  \cE(\qn, \varphi)  + \partial_\varphi \cE(\qn, \varphi)(\Delta \varphi)= 0
\end{equation}
where $\qn = q + \Delta q$. If $\phi = \phi(z, \bar{z})$, we define
\[
  \phi^-(z, \bar{z}) = \phi(\lambda^{-1} z, \lambda \bar{z}), \quad
  \phi^+(z, \bar{z}) = \phi(\lambda z, \lambda^{-1} \bar{z})
\]
and
  \[
    \nabla \phi = \phi^- - \lambda^{-1}\phi, \quad \nabla^+ \phi = \phi - \lambda \phi^+.
  \]
\begin{lemma}[See \cite{LeviMoserLagrangian2001}] \label{lem:ML}
  Let $h = \partial_{12} S(\varphi^-, \varphi) (\varphi_z)(\varphi_z)^-$, then for $w = w(z, \bar{z})$:
  \begin{equation}  \label{eq:S-varphi}
	\partial_\varphi \cE (w) \varphi_z =  \partial_z \cE \cdot w + \nabla^+(h \nabla (w/ \varphi_z)). 
  \end{equation}
\end{lemma}

Let us first record the following lemma, whose proof is straightforward.
\begin{lemma}
\[
  (\phi^-)_z = \lambda^{-1} (\phi_z)^-, \quad
  (\phi^+)_z = \lambda (\phi_z)^+, 
\]
\[
  (\phi^-)_{\bar{z}} = \lambda (\phi_z)^-, \quad
  (\phi^+)_{\bar{z}} = \lambda^{-1} (\phi_{\bar{z}})^+.
\]
\end{lemma}
For simplicity, $(\phi_{z})^\pm$ will be denoted $\phi_z^\pm$, and similarly for $\phi_{\bar{z}}^\pm$.

\begin{proof}[Proof of Lemma \ref{lem:ML}]
  We have
  \[
    \begin{aligned}
      \partial_\varphi \cE(w)
  &= \partial_\varphi (\partial_2 S(\varphi^-, \varphi) + \partial_1 S(\varphi, \varphi^+))(w) \\
  & = \partial_{12} S (w^-) + \partial_{22} S (w)
  + \partial_{11} S^+ (w) + \partial_{12} S^+ (w^+),
    \end{aligned}
  \]
  where we used the notational convention that $S = S(\varphi^-, \varphi)$ and $S^+ = S(\varphi, \varphi^+)$. 

  Similarly, 
  \[
    \begin{aligned}
      \partial_z \cE
  &= \partial_z (\partial_2 S(\varphi^-, \varphi) + \partial_1 S(\varphi, \varphi^+)) \\
  & = \partial_{12} S (\varphi^-)_z + \partial_{22} S  (\varphi_z) 
  + \partial_{11} S^+ (\varphi_z) + \partial_{12} S^+ (\varphi^+)_z \\
  & = \lambda^{-1} \partial_{12} S \varphi_z^- + \partial_{22} S  \varphi_z 
  + \partial_{11} S^+ \varphi_z + \lambda \partial_{12} S^+ \varphi_z^+.
    \end{aligned}
  \]
Then
  \begin{equation}  \label{eq:E-varphi-z}
    \begin{aligned}
 & \quad \partial_\varphi \cE(w) \varphi_z - 
 \partial_z \cE (w)  \notag\\
 & = \partial_{12}S \left( (w^-) \varphi_z - \lambda^{-1} w \varphi_z^-\right) 
 + \partial_{12}S^+ \left( (w^+) \varphi_z - \lambda w \varphi_z^+\right) \notag\\
 & = \partial_{12}S (\varphi_z \varphi_z^-) \left( w^-/ \varphi_z^- - \lambda^{-1} w/ \varphi_z\right)  \notag\\
 & \quad - \partial_{12}S^+ (\varphi_z \varphi_z^+) \lambda \left(  w/ \varphi_z - \lambda^{-1} w^+ /\varphi_z^+ \right) \\
 & =  \nabla^+\left( \partial_{12} S (\varphi_z \varphi_z^-) \nabla (w/ \varphi_z)\right).
    \end{aligned}
  \end{equation}
\end{proof}

Using Lemma \ref{lem:ML}, we multiply \eqref{eq:lin-eq} by $\varphi_z$, and drop the quadratically small term $\partial_z \cE (\Delta \varphi)$ to get the following equation
\begin{equation}  \label{eq:app-lin}
  \cL_z(\Delta \varphi) := \nabla^+(h \nabla(\Delta \varphi/\varphi_z)) = - \cE(\qn, \varphi) \varphi_z
\end{equation}
which we will call the \emph{cohomological equation}. (The name $\cL_z$ is related to the $z$ derivative used to derive Lemma \ref{lem:ML}). Since $\cL_z$ is a second order operator, we need to solve \eqref{eq:app-lin} in two steps
\[
\begin{aligned}
  & \nabla^+(\psi)  = - \cE(\qn, \varphi) \varphi_z, \\
  & h \nabla(\Delta\varphi/ z)  = \psi. 
\end{aligned}
\]
Let's call the first equation the \emph{outer equation} and the second the \emph{inner equation}.

\section{Solving the outer cohomological equation}\label{section:Solving_the_outer_cohomological_equation}

For $\phi = \sum_{j, k} \phi_{j, k} z^j \bar{z}^k$, we have
\[
  \nabla^+ \phi = \sum_{j, k}\phi_{j, k}(1 - \lambda^{j - k + 1}) z^j \bar{z}^k, 
\]
Similarly, 
\[
  \nabla \phi = \sum_{j, k}\phi_{j,k} (\lambda^{k - j} - \lambda^{-1}) z^j \bar{z}^k, 
\]
We define further more the kernels the above operators, namely 
\begin{align*}
    K^+ :=& \ker \nabla^+ = \Span\{ z^j \bar{z}^{j + 1} \st j \ge 0\}\\
    K :=& \ker \nabla = \Span\{ z^{j+1} \bar{z}^j \st j \ge 0\},
\end{align*}
 as well as the orthogonal complements of 
the kernels of the standard projections:
\begin{align*}
    (K^+)^{\perp} :=& (\ker \nabla^+)^c = \Span\{ z^k \bar{z}^{j} \st j-k\neq   1, j,k\geq 0\}\\
    (K)^{\perp} :=& (\ker \nabla)^c = \Span\{ z^k \bar{z}^{j} \st k-j\neq   1, j,k\geq 0\}\\
\end{align*}

The operators $\nabla^+$ and $\nabla$ are invertible on  $(K^+)^\perp$ and $K^\perp$,  respectively. The inverses are given by 
\begin{equation}  \label{eq:E}
  E^+(\phi) = \sum_{k \ne j + 1} \frac{\phi_{j, k}}{1 - \lambda^{j - k + 1}} z^j \bar{z}^k, \quad
  E(\phi) = \sum_{j \ne k + 1} \frac{\phi_{j, k}}{\lambda^{k - j} - \lambda^{-1}} z^j \bar{z}^k.
\end{equation}
Let us define
\[
  \left[\sum_{j, k} \phi_{j, k} z^j \bar{z}^k \right] = \sum_{j = 1}^\infty \phi_{j, j} z^j \bar{z}^j. 
\]
Note that if we write $\phi(z, \bar{z})$ in polar coordinates $(r, \theta)$, then
\[
  [\phi] = \frac{1}{2\pi}\int_0^{2\pi} \phi(r, \theta) d\theta. 
\]
Using this notation, the projections to $K^+$ and $K$ are given by
\[
  \Pi_+(\phi) = [z \phi]/z, \quad \Pi(\phi) = [\bar{z} \phi]/\bar{z}.
\]

The following properties of the $[\cdot]$ operator are easy to prove.
\begin{lemma}\label{lem:ave-property} Let $\kappa$, $\phi$ be power series in $z, \bar{z}$, then:
\begin{enumerate}
 \item $[\phi^+] = [\phi^-] = [\phi]$.
 \item If $[\kappa] = \kappa$, then $[\kappa \phi] = \kappa [\phi]$.
 \item $\Pi_+(\nabla^+ \phi)=\Pi(\nabla \phi)=0$
\end{enumerate}	
\end{lemma}
\begin{proof}
\begin{enumerate}
    \item By definition of $\phi^{\pm}$, one has that 
\begin{align*}
[\phi^+] = \sum_{j=1}^{\infty} \phi_{jj} (\lambda z^j) (\lambda^{-1} \bar{z}^{j}) =      \sum_{j=1}^{\infty} \phi_{jj} z^j  \bar{z}^{j} = [\phi]\\
[\phi^-] = \sum_{j=1}^{\infty} \phi_{jj} (\lambda^{-1} z^j) (\lambda \bar{z}^{j}) =      \sum_{j=1}^{\infty} \phi_{jj} z^j  \bar{z}^{j}= [\phi]
\end{align*}
\item Let $\displaystyle \kappa = [\kappa] = \sum_{l=1}^{\infty}\kappa_{ll}z^l\bar{z}^l$, it follows that 
\begin{align*}
    [\kappa \phi] = [\sum_{l=1}^{\infty}\kappa_{ll}z^l\bar{z}^l\sum_{k,j}\phi_{k,j}z^k\bar{z}^{j}] = \sum_{p=2}^{\infty} \sum_{l+j=p} \kappa_{ll} \phi_{jj} = [\kappa][\phi]
\end{align*}
\item By definition of $\nabla^+ \phi$
\[\Pi_+(\nabla^+ \phi) =\frac{1}{z}\left[\sum_{j, k}\phi_{j, k}(1 - \lambda^{j - k + 1}) z^{j+1} \bar{z}^k\right]=0, \]

\[\Pi(\nabla \phi) =\frac{1}{\bar{z}} \left[\sum_{j, k}\phi_{j, k}(\lambda^{k-j} - \lambda^{-1}) z^{j} \bar{z}^{k+1}\right]=0
\]
\end{enumerate}

\end{proof}
The operators we defined is closely related to the symmetries of the system. Recall that $\phi_{j,k}=\phi_{k,j}$ in the formal power series $\phi$. Define
\begin{equation}  \label{eq:I}
 I(z, \bar{z}) = (\bar{z}, z), 	
\end{equation}
We have the following lemma, whose proof is straightforward.
\begin{lemma}\label{lem:phi-symmetry}
  Let $\phi = \phi(z, \bar{z})$ be a formal power series. Then:
\begin{enumerate}
 \item $(\phi^-) \circ I = (\phi \circ I)^+$, $(\phi^+) \circ I = (\phi \circ I)^-$.
 \item $(\nabla \phi) \circ I = -\lambda^{-1} \nabla^+(\phi \circ I)$, $(\nabla^+ \phi) \circ I = - \lambda \nabla(\phi \circ I)$.
 \item $E^+(\phi) \circ I = - \lambda^{-1} E(\phi \circ I)$, $E(\phi) \circ I = - \lambda E^+(\phi \circ I)$.
 \item $\phi_z \circ I = (\phi \circ I)_{\bar{z}}$, $\phi_{\bar{z}} \circ I = (\phi \circ I)_z$.
\end{enumerate}	
\end{lemma}
\begin{proof}
Let $\phi(z,\bar{z})= \sum_{i,j}\phi_{ij}z^i\bar{z}^{j}$, then  
 \begin{enumerate}
 \item By definition, 
\begin{align*}
   (\phi^{-})\circ I &= \phi(\lambda^{-1}\bar{z}, \lambda z  )= (\phi\circ I)^{+} \\ 
   (\phi^{+})\circ I &= \phi(\lambda\bar{z}, \lambda^{-1}z  ) = (\phi\circ I)^{-} 
\end{align*}
\item $\nabla \phi= \phi^{-}-\lambda^{-1}\phi, \nabla^{+} \phi=\phi-\lambda \phi^{+}$. Hence 
\begin{align*}
    (\nabla \phi)\circ I &= \phi^{-}\circ I -\lambda^{-1}\phi\circ I = -\lambda^{-1}(\phi\circ I - \lambda \phi^{-}\circ I) =-\lambda^{-1}(\phi\circ I - \lambda (\phi\circ I)^{+}) \\
    &=-\lambda^{-1}\nabla^{+}(\phi\circ I)\\
    (\nabla^{+} \phi)\circ I &= \phi\circ I -\lambda\phi^{+}\circ I = -\lambda(\phi^{+}\circ I - \lambda^{-1} \phi\circ I)  = -\lambda((\phi\circ I)^{-} - \lambda^{-1} \phi\circ I) \\
    &= -\lambda \nabla(\phi\circ I)
\end{align*}
\item   $\displaystyle E^+(\phi) = \sum_{k \ne j + 1} \frac{\phi_{j, k}}{1 - \lambda^{j - k + 1}} z^j \bar{z}^k,   E(\phi) = \sum_{j \ne k + 1} \frac{\phi_{j, k}}{\lambda^{k - j} - \lambda^{-1}} z^j \bar{z}^k$. Hence 
\begin{align*}
    E^+(\phi)\circ I &= \sum_{k \ne j + 1} \frac{\phi_{j, k}}{1 - \lambda^{j - k + 1}} \bar{z}^j z^k = -\lambda^{-1} \sum_{k \ne j + 1} \frac{\phi_{j, k}}{ \lambda^{j - k } -\lambda^{-1}} \bar{z}^j z^k = -\lambda^{-1} E(\phi\circ I)\\
    E(\phi)\circ I &= \sum_{j \ne k + 1} \frac{\phi_{j, k}}{\lambda^{k - j} - \lambda^{-1}} \bar{z}^j z^k = -\lambda \sum_{k \ne j + 1} \frac{\phi_{j, k}}{1- \lambda^{j - k+1 } } \bar{z}^j z^k = -\lambda E(\phi\circ I)
\end{align*}
\item Direct calculation shows that 
\begin{align*}
    \phi_z\circ I = \sum_{j,k} j \bar{z}^{j-1} z^{k} = (\phi\circ I)_{\bar{z}} \\
    \phi_{\bar{z}} \circ I = \sum_{j,k} k \bar{z}^{j} z^{k-1} = (\phi\circ I)_{z} 
\end{align*}
 \end{enumerate}
\end{proof}
Noting that $S(t_1, t_2) = S(t_2, t_1)$, we also have:
\begin{lemma}\label{lem:ES-symmetry}
Suppose $\varphi \circ I = \varphi$, then:
\begin{enumerate}
  \item $\partial_2 S(\varphi^-, \varphi) \circ I = \partial_1 S(\varphi, \varphi^+)$, $\partial_1 S(\varphi, \varphi^{+}) \circ I = \partial_2 S(\varphi^{-}, \varphi)$.
  \item $\partial_{12} S(\varphi^-, \varphi) \circ I = \partial_{12} S(\varphi, \varphi^+)$.
  \item $\partial_{22}S(\varphi^{-},\varphi)\circ I = \partial_{11}S(\varphi,\varphi^{+}) $. 
 \item  $\cE(q, \varphi) \circ I = \cE(q, \varphi)$.
\end{enumerate}	
\end{lemma}
\begin{proof}
Since $S(t_1, t_2) = S(t_2, t_1)$, we have $\partial_1 S(t_1, t_2) = \partial_2 S(t_2, t_1)$. Therefore
\[
  \partial_1 S(\varphi^-, \varphi) \circ I = \partial_1 S(\varphi^- \circ I, \varphi \circ I) = \partial_2 S(\varphi \circ I, \varphi^- \circ I) = \partial_2 S(\varphi \circ I,(\varphi \circ I)^+)  = \partial_2 S(\varphi, \varphi^+), 
\]
where we used item 1 of Lemma~\ref{lem:phi-symmetry} and $\varphi \circ I = \varphi$. This proves the first part of item 1. The second part follows from a symmetric computation. 

For item 2, note that $\partial_{12} S(t_1, t_2) = \partial_{21} S(t_2, t_1)$,
\[
  \partial_{12} S(\varphi^-, \varphi) \circ I 
   = \partial_{12} S(\varphi^- \circ I, \varphi \circ I) 
  = \partial_{12} S(\varphi \circ I, \varphi^- \circ I) 
   = \partial_{12} S(\varphi, \varphi^+)
\]
For item 3, again by symmetry, 
\[
\partial_{22} S(\varphi^-, \varphi) \circ I = \partial_{22} S(\varphi^-\circ I, \varphi\circ I ) =  \partial_{11} S(\varphi\circ I, \varphi^-\circ I) =  \partial_{11} S(\varphi, \varphi^{+})
\]
For item 4, recall that by definition
\[
    \cE(q, \varphi) := \partial_2 S(\varphi^-, \varphi) + \partial_1 S(\varphi, \varphi^+),
\]
Then using the previous result one sees that  
\[
    \cE(q, \varphi)\circ I = \partial_2 S(\varphi^-, \varphi)\circ I + \partial_1 S(\varphi, \varphi^+) \circ I = \partial_1 S(\phi,\phi^{+})+\partial_2 S(\phi^{-},\phi)=  \cE(q, \varphi).
\]
\end{proof}

We now solve the cohomological equation \eqref{eq:app-lin}. Note that the equation only has a solution if $-\mathcal{E}(q^{*},\phi)\phi_z$ is in $(K^+)^\perp$. On the \emph{formal} level, we can always choose $\Delta q$ so that this is the case, which is justified by the following lemma. 
\begin{proposition}\label{prop:outer-eq}
  Suppose $\varphi \circ I = \varphi$. Then there exists $\Delta q(t) = \sum_{k = 2}^\infty \eta_{2k} t^{2k}$ such that both
  \begin{equation}  \label{eq:zero-mean-outer}
    \Pi_+\left( \cE(q+\Delta q, \varphi) \varphi_z\right) = 0
  \end{equation}
  and 
  \begin{equation}  \label{eq:zero-mean-outer-dual}
	\Pi\left( \cE(q+\Delta q, \varphi) \varphi_{\bar{z}}\right) = 0
  \end{equation}
  hold.
\end{proposition}
Note that \eqref{eq:zero-mean-outer} and \eqref{eq:zero-mean-outer-dual} are equivalent by Lemmas \ref{lem:phi-symmetry} and \ref{lem:ES-symmetry}.
\begin{lemma}\label{lem:S-zero-mean}
  For a symmetric and normalised $\varphi$, equation \eqref{eq:zero-mean-outer} and \eqref{eq:zero-mean-outer-dual} are both equivalent to 
  \begin{equation}  \label{eq:S-zero-mean}
    [S_{\qn}(\varphi^-, \varphi)] = 1. 
\end{equation}
\end{lemma}
\begin{proof}

Observe that 
\[
    \partial_2 S(\varphi^-, \varphi) \varphi_z
    = \partial_z\left( S(\varphi^-, \varphi)\right) - \lambda^{-1} \partial_1 S(\varphi^-, \varphi) \varphi^-_z,
\]
hence
\begin{equation}  \label{eq:S-to-E}
    \begin{aligned}
      \cE \varphi_z & = \partial_2 S \varphi_z + \partial_1 S^+ \varphi_z
      = \partial_z S - \lambda^{-1}(\partial_1 S \varphi_z^- - \lambda \partial_1 S^+ \varphi_z) \\
                    & = \partial_z S - \lambda^{-1} \nabla^+ (\partial_1 S \varphi_z^-).
    \end{aligned}
\end{equation}
Similarly, we have
\begin{equation}  \label{eq:S-to-E-bar}
  \cE \varphi_{\bar{z}} = \partial_{\bar{z}}S^+ +  \nabla^+ (\partial_2 S^+ \varphi_{\bar{z}}^+).
\end{equation}
This allows us to verify a different condition which implies \eqref{eq:zero-mean-outer}.

Let $f = \displaystyle \sum_{n = 0}^\infty f_n (z \bar{z})^n$, define the operator
\begin{equation}  \label{eq:D}
  D(f) = \sum_{n = 1}^\infty n f_n (z \bar{z})^n.
\end{equation}
Then 
\begin{equation}  \label{eq:zS}
  [z \partial_z S] = D([S]).
\end{equation}
In particular $[z \partial_z S] = 0$ if and only if $[S] = \const$. Moreover, since $S(\varphi^-, \varphi) = 1 + O_2(z, \bar{z})$, $[S] = \const$ is the same as $[S] = 1$.
Let 
\begin{equation}  \label{eq:D-bar}
  \bar{D}(f) = \bar{D}(f - f_0) = \sum_{n = 1}^\infty \frac{f_n}{n} (z \bar{z})^n 	
\end{equation}
then $ \bar{D}(D(f))=f-f_0$. Moreover, by \eqref{eq:zS} and lemma \ref{lem:ave-property},
\begin{align} \label{eq:D-proj}
   D([S])&=[z\partial_z S]= [z\mathcal{E}\phi_z+ \lambda^{-1}z\nabla^+(\partial_1S\phi^-_z)] =z\Pi_+(\mathcal{E}\varphi_z)+\lambda^{-1}\Pi_+(\nabla^+(\partial_1 S\varphi^-_z) \notag\\
   &=z\Pi_+(\mathcal{E}\varphi_z)
\end{align}
It follows that 
\begin{align}  \label{eq:S-to-E2}
  [S] - 1 &= \bar{D}(D([S]))= \bar{D}(z \Pi_+(\cE \varphi_z)). 
\end{align}
The result then follows from \eqref{eq:S-to-E}.
\end{proof}
With the lemma in hand, we now prove proposition \ref{prop:outer-eq}.
\begin{proof}[Proof of Proposition \ref{prop:outer-eq}]
  It suffices to prove \eqref{eq:S-zero-mean}. Recall that
  \[
    S(t_1, t_2)  = 1 + \frac14 q_2 (t_1 + t_2)^2 + \frac14 p_2(t_1 - t_2)^2 + O_4(t_1, t_2).
  \]
  Denote by $Q(t_1, t_2)$ the quadratic part of $S$, and using $\varphi = \varphi^{(0)} + O_3 = z + \bar{z} + O_3 $, we have
  \[
    \begin{aligned}
     \  [Q(\varphi^-, \varphi)] & = [Q((\varphi^{(0)})^-, \varphi^{(0)})] + O_4 \\
                              &= \frac14\left[ q_2(\lambda^{-1} z + \lambda \bar{z}  + z + \bar{z})^2
                              + p_2  (\lambda^{-1} z + \lambda \bar{z} - z - \bar{z})^2\right] + O_4  \\
                              & = \frac14 \left( 2 q_2 (\lambda^{-1} + 1)(\lambda + 1)
                              + 2 p_2 (\lambda^{-1} - 1)(\lambda - 1)\right) z \bar{z} + O_4  \\
							  & =  \frac{\lambda^{-1}}{2} \left( 
							   q_2(\lambda + 1)^2 - p_2(\lambda - 1)^2	
							  \right)z\bar{z} + O_4 
							  = O_4, 
    \end{aligned}
  \]
  where the last step is due to \eqref{eq:q2}. We conclude that 
  \[
    [S_q(\varphi^-, \varphi) ]= O_4 
  \]
  for any $\varphi = (z + \bar{z}) + O_3$.

  Write $\Delta q = \sum_{k = 2}^\infty \eta_{2k} t^{2k}$ and denote $\xi = \varphi^- + \varphi$, $\zeta = \varphi^- - \varphi$, $\xi_0 = (\varphi^{(0)})^- + \varphi^{(0)}$,  \eqref{eq:S-zero-mean} becomes
  \begin{equation}  \label{eq:S-eq}
	\sum_{k = 2}^\infty \eta_{2k} [\xi^{2k} p(\zeta)] + [S_q] = 0.
  \end{equation}
  Suppose $[\xi^{2k} p(\zeta)] = \sum_{j = k}^\infty P_{j, k} (z \bar{z})^j$, \eqref{eq:S-eq} is equivalent to 
  \begin{equation}  \label{eq:inf-matrix}
	\sum_{k = 2}^j P_{j, k} \eta_{2k} = - [S_q]_{2j}, \quad j \ge k.
  \end{equation}
  which is an infinite, lower diagonal linear system, with the infinite coefficient matrix being
  \begin{align*}
  \begin{pmatrix}
  0 & 0 & 0 & 0 &0 &\cdots \\
  0 & P_{22} & 0 & 0&0 &\cdots \\
  0 & P_{32} & P_{33} & 0&0 &\cdots\\
  0 & P_{42} & P_{43} & P_{44} & 0 &\cdots \\
  \cdots & \cdots & \cdots & \cdots & \cdots & \cdots  
  \end{pmatrix}
  \end{align*}
  Since
  \[
	P_{j, j} = [\xi_0^{2j}]_{2j}  = \left[ ((\lambda^{-1} - 1)z + (\lambda - 1) \bar{z})^{2j}\right]_{2j} = 
    \binom{2j}{j} (\lambda^{-1} - 1)^j(\lambda - 1)^j \ne 0,
  \]
  equation \eqref{eq:inf-matrix} has a unique solution $(\eta_{2k})_{k \ge 2}$, given by the recursive formula 
  \begin{align}
    \eta_{2k} = \begin{cases}    
      \displaystyle -\frac{1}{P_{22}}[S_q]_4, & j=2\\
      \displaystyle -\frac{1}{P_{jj}}\left([S_q]_{2j} - \sum_{k=2}^{j-1}P_{j,k}\eta_{2k}\right), &j>2
      \end{cases}
  \end{align}
\end{proof}

While Lemma \ref{lem:S-zero-mean} provides a satisfactory answer on the formal level, the operator mapping $[S_q]$ to $\Delta q$ via \eqref{eq:S-zero-mean} is hopelessly unbounded, even if we use weighted norms. For the actual iteration, we will truncate equation \eqref{eq:S-zero-mean} to a sufficiently high order. This means \eqref{eq:S-zero-mean} is only approximately satisfied.

Given $M \in \N$, define the truncation operator
\begin{equation}  \label{eq:truncation}
  \Lambda_M(q) = \sum_{j \le M} q_j t^j, \quad
  \Lambda_M(\phi) = \sum_{j + k \le M} \phi_{j, k} z^j \bar{z}^k. 
\end{equation}
\begin{corollary}\label{cor:Delta-q}
  For any $M \ge 2$, there exists a unique polynomial $\Delta q = \sum_{k = 2}^M \eta_{2k} t^{2k}$, such that 
  \begin{equation}  \label{eq:Delta-q}
	\Lambda_{2M} \left( [S_{q + \Delta q}]\right) = 1.
  \end{equation}
  Moreover, if $\Lambda_{2N}([S_q] - 1) = 0$ for some $N < M$, we have 
  \[
	\Lambda_{2N} \Delta q = 0.
  \]
\end{corollary}
The proof of Corollary~\ref{cor:Delta-q} is the same as Proposition \ref{prop:outer-eq} except that we invert a finite matrix. 

When applying the iterative process to prove the main theorem, we will choose an appropriate $M \in \N$ for each step, then define $\Delta q$ using Corollary \ref{cor:Delta-q}. 

Set
\begin{equation}\label{eq:sol-outer-part}
   \psi = -E^+\Bigl( \cE(q^*, \varphi) \varphi_z - \Pi_+(\cE(q^*, \varphi) \varphi_z) \Bigr),   
\end{equation}

then $\psi$ solves the \emph{outer part} of \eqref{eq:app-lin} up to the error term $\Pi_+(\cE(q^*, \varphi) \varphi_z)$ (the error term can be made zero for a formal $\Delta q$). Plug into \ref{eq:app-lin}, we need to solve the \emph{inner part} of the equation
\[
  h\nabla (w/\varphi_z) = \psi,
\]
which requires $\Pi(\psi/h) = 0$. This does not hold in general. However, we will show that if $[S_{\qn}]$ is sufficiently close to $1$, $\Pi(\psi/h)$ is also small. We can then set
\[
  \nabla(w/ \varphi_z) = \psi/h - \Pi(\psi/h).
\]
This is done in the next section.

\section{Solving the inner cohomological equation}

In this section, we will show that $\Pi(\psi/h)$ is small if $[S_{\qn}] - 1$ is sufficiently small. Since the actual result of this section is quite technical, let us state a \emph{formal} version as motivation. In proposition \ref{prop:small-average-psi} below, we estimate the amplitude of $\Pi(\frac{\psi}{h})$ with $\Delta q$ constructed as in section \ref{section:Solving_the_outer_cohomological_equation} without any truncation error, which in turn gives the solution of Trechev's original series. In this case, equation \eqref{eq:sol-outer-part} is simplified to $\psi = E^+\left( \cE(\qn, \varphi) \varphi_z\right)$. 

\begin{proposition}\label{prop:small-average-psi}
  Given $(q, \varphi)$, suppose $\qn = q + \Delta q$, $\psi$ are chosen as follows:
  \[
	[S(\qn, \varphi)] = 1, \quad
	\psi = E^+\left( \cE(\qn, \varphi) \varphi_z\right). 
  \]
  Then
  \[
	\Pi\left( \frac{\psi}{h}\right) 
  \]
  is quadratically small with respect to $\cE$. 
\end{proposition}

First, let us note the following calculation:
\[
 \begin{aligned}
   \partial_z \cE \cdot \varphi_{\bar{z}} - \partial_{\bar{z}} \cE \cdot \varphi_z 
   &=\left(\lambda^{-1} \partial_{12} S \varphi_z^- + \partial_{22} S  \varphi_z 
  + \partial_{11} S^+ \varphi_z + \lambda \partial_{12} S^+ \varphi_z^+\right)\varphi_{\bar{z}}\\
  &\quad \quad -\left(\lambda \partial_{12} S \varphi_{\bar{z}}^- + \partial_{22} S  \varphi_{z} 
  + \partial_{11} S^+ \varphi_{\bar{z}} + \lambda^{-1} \partial_{12} S^+ \varphi_{\bar{z}}^+ \right)\varphi_{\bar{z}} \\
   &= \partial_{12} S\left( \lambda^{-1} \varphi_z^- \varphi_{\bar{z}}
   - \lambda \varphi_{\bar{z}}^- \varphi_z \right) -
   \partial_{12} S^+\left( 
    \lambda^{-1} \varphi_{\bar{z}}^+ \varphi_z-\lambda \varphi_z^+ \varphi_{\bar{z}} \right) \\
   &= \kappa - \kappa^+
\end{aligned}
\]
where
\begin{equation}  \label{eq:kappa}
	\kappa = \partial_{12} S\left( \lambda^{-1} \varphi_z^- \varphi_{\bar{z}}
   - \lambda \varphi_{\bar{z}}^- \varphi_z \right).
\end{equation}
For $\phi = \sum_{j, k} \phi_{j, k} z^j \bar{z}^k$, 
\[
  \phi - \phi^+ = \sum_{j \ne k} \phi_{j, k}(1 - \lambda^{j - k}) z^j \bar{z}^k,
\]
the operator
\begin{equation}  \label{eq:tilde-E}
  \tilde{E}(\phi) = \sum_{j \ne k} \phi_{j, k}/(1 - \lambda^{j - k}) z^j \bar{z}^k
\end{equation}
satisfies
\[
  \tilde{E}(\phi - \phi^+) = \phi - [\phi].
\]
It follows that 
\begin{equation}  \label{eq:non-res-kappa}
  \kappa - [\kappa] = \tilde{E} \left( \partial_z \cE \cdot \varphi_{\bar{z}} - \partial_{\bar{z}} \cE \cdot \varphi_z \right)
\end{equation}
has the same order as $\cE$ does. The function $\kappa$ has a geometrical meaning: it is the determinant of the Jacobian matrix of the mapping
\[
  (z, \bar{z}) \mapsto \left( \varphi^-, \partial_1 S(\varphi^-, \varphi)\right).
\]
This mapping is the symplectic version of the coordinate change $(z, \bar{z}) \mapsto (\varphi^-, \varphi)$ ($\partial_1 S(\varphi^-, \varphi)$ is the momentum variable), hence $\kappa$ is the conformal factor of our coordinate change. Our calculation suggests that if $\cE = 0$, then $\kappa$ depends only on the radial component $z \bar{z}$.

\begin{proof}[Proof of Proposition~\ref{prop:small-average-psi}]
Rewrting $h$ using $\kappa$, we get
\[
  \begin{aligned}
    \frac{\psi}{h} & = \frac{\psi}{\partial_{12} S \varphi_z^- \varphi_z}
  = \psi \frac{\lambda^{-1} \varphi_{\bar{z}}/\varphi_z - \lambda \varphi_{\bar{z}}^-/\varphi_z^-}{\kappa}.
  \end{aligned}
\]
Denote $g = \varphi_{\bar{z}}/\varphi_z$. Lemma \ref{lem:ave-property} implies that $[(\bar{z} \psi g^{-})] = [(\bar{z} \psi g^{-})^{+}] = [\lambda^{-1}\bar{z}\psi^{+}g]$, hence 
\[
  \begin{aligned}
  \left[\bar{z}  \frac{\psi (\lambda^{-1} g - \lambda g^-)}{[\kappa]}\right]
  & = \frac{1}{[\kappa]} \left( \lambda^{-1} [\bar{z} \psi g] - \lambda [\bar{z} \psi g^-]\right)
  = \frac{1}{[\kappa]} \left( \lambda^{-1} [\bar{z} \psi g] - \lambda [\lambda^{-1} \bar{z} \psi^+ g ]\right) \\
  & = \frac{1}{\lambda[\kappa]} \left( [\bar{z} \psi g] - \lambda [\bar{z} \psi^{+} g] \right) = \frac{1}{\lambda [\kappa]} \left[ \bar{z} g \nabla^+\psi\right] = \frac{1}{\lambda [\kappa]} \left[ \bar{z} \frac{\varphi_{\bar{z}}}{\varphi_z} \cE \varphi_z\right] \\
  & = \frac{\left[ \bar{z} \cE \varphi_{\bar{z}}\right]}{\lambda[\kappa]} 
  = \frac{\bar{z}}{\lambda[\kappa]} \Pi\left( \cE \varphi_{\bar{z}}\right) = 0,
  \end{aligned}
\]
where the last equality is due to the way in which $\Delta q$ is constructed and the dual version of Lemma \ref{lem:S-zero-mean}.

As a result, we get
\begin{align}
    \Pi\left( \frac{\psi}{h}\right)
  &= \frac{1}{\bar{z}} \left[ \bar{z} \frac{\psi(\lambda^{-1} g - \lambda g^-)}{\kappa}\right]\notag \\
 &=\frac{1}{\bar{z}} \left[ \bar{z} \frac{\psi(\lambda^{-1} g - \lambda g^-)}{\kappa}\right] -\frac{1}{\bar{z}} \left[ \bar{z} \frac{\psi(\lambda^{-1} g - \lambda g^-)}{[\kappa]}\right]\notag\\
  &= \frac{1}{\bar{z}} \left[ \bar{z} \psi(\lambda^{-1} g - \lambda g^-)(1/\kappa - 1/[\kappa]) \right] = O_2(\mathcal{E}),   
\end{align}
i.e., $\displaystyle \Pi\left( \frac{\psi}{h}\right)$is quadratically small in $\cE$, since $\psi = E^+(\cE \varphi_z)$ and $(1/\kappa - 1/[\kappa]) = - ([\kappa] - \kappa)/(\kappa [\kappa])$ are both of order $\cE$ (see \eqref{eq:non-res-kappa}).
\end{proof}

We now state a proposition that as an analogue of the previous result, yet the assumption $[S_{\qn}] = 1$ is dropped. It turns out the quantity $ \Pi\left( \psi/h\right)$ can be decomposed into two parts: the first part is generated due to the deviation of $[S_q]$ from 1, and the second part is again quadratically small in $\mathcal{E}$. To this end, we  we will use the notation for $N \in \N$,
\[
  f = O_{N + 1} \iff \Lambda_N (f) = 0. 
\]

\begin{proposition}\label{prop:S-2nd}
For
\[
  \psi = -E^+\left( \cE(\qn, \varphi) \varphi_z - \Pi_+(\cE(\qn, \varphi) \varphi_z)\right), 
\]
we have
\[
  \Pi\left( \psi/h\right) 
  = R_1 + R_2, 
\]
where 
\begin{align}
\label{eq:r1andr2}
R_1 = \frac{ -[\bar{z} \partial_{\bar{z}} S] + [\bar{z} g [z \partial_z S]/ z]}{ \lambda z [\kappa]}, \quad 
R_2 =  \frac{1}{z} \left[\bar{z} \psi(\lambda^{-1} g - \lambda g^-) \frac{[\kappa] - \kappa}{\kappa [\kappa]}\right]. 
\end{align}

Moreover, if 
\[
  [S_{\qn}] - 1 = O_{2M + 2}, \quad 
  \cE(\qn, \varphi) = O_{2N + 1}
\]
for some $M , N \in \N$, we have
\[
  R_1 = O_{2M + 1}, \quad R_2 = O_{4N + 1}.
\]
\end{proposition}
\begin{remark}
  The $R_1$ is small if $[S]$ is close to $1$, $R_2$ is quadratically small in $\cE$.	
\end{remark}
\begin{proof}
  We repeat of the proof of Proposition~\ref{prop:small-average-psi}, this time without assuming $[S_{q + \Delta q}] = 1$. Let us recall \eqref{eq:S-to-E} and \eqref{eq:S-to-E-bar}:
  \[
	\cE \varphi_z = \partial_z S - \lambda^{-1} \nabla(\partial_1 S \varphi_z^-), \quad
	\cE \varphi_{\bar{z}} = \partial_{\bar{z}} S^+ + \nabla^+ (\partial_2 S \varphi_{\bar{z}}), 
  \]
  which implies
  \[
	[z \partial_z S] = z \Pi_+(\partial_z S) = z \Pi_+(\cE \varphi_z) = [z \cE \varphi_z], \quad
	[\bar{z} \partial_{\bar{z}} S] = [\bar{z} \cE \varphi_{\bar{z}}].
  \]
  We have
  \[
	\begin{aligned}
  \left[\bar{z}  \frac{\psi (\lambda^{-1} g - \lambda g^-)}{[\kappa]}\right]
  & = \frac{1}{\lambda [\kappa]}  \left[ \bar{z} g \nabla^+ \psi\right]
   = \frac{1}{\lambda [\kappa]}  \left[- \bar{z} g (\cE \varphi_z + \Pi_+ (\cE \varphi_z)) \right] \\
  & = \frac{1}{\lambda [\kappa]} \left(- [\bar{z} \cE \varphi_{\bar{z}}] 
   + \bar{z} g \Pi_+(\cE \varphi_z)\right)  \\
  & = \frac{1}{\lambda [\kappa]} \left(- [\bar{z} \partial_{\bar{z}} S] 
	 +[\bar{z} g [z \partial_z S]/ z]   \right) =  z R_1.
	\end{aligned}
  \]
  We have
  \[
  \Pi\left( \frac{\psi}{h}\right)
= \frac{1}{z} \left[ \bar{z}\psi(\lambda^{-1} g - \lambda g^-) ([\kappa]^{-1} + \kappa^{-1} - [\kappa]^{-1})\right]
  = R_1 + R_2. 
  \]

  We now prove the ``moreover'' part. If $[S] = O_{2M + 2}$, then $[z \partial_z S] = D([S]) = O_{2M +2}$. Noting that $\kappa$, $g$ both have nonzero constant term, we have
$R_1 = O_{2M + 1}$.

For $R_2$, we note that the operators $\Pi_+$, $E^+$ both preserve order. Therefore $\psi = O_{2N + 1}$. Since $\kappa - [\kappa]$ is related to $\cE$ after taking one derivative  (see \eqref{eq:non-res-kappa}), we get $\kappa - [\kappa] = O_{2N}$. It's then straightforward to see that $R_2 = O_{4N + 1}$. 
\end{proof}

Finally, let us summarize our approach to solve equation \eqref{eq:lin-eq}. After $\qn = q + \Delta q$ is chosen, we set
\begin{equation}  \label{eq:solution-coh}
  \psi = -E^+\left( \cE(\qn, \varphi) \varphi_z - \Pi_+(\cE(\qn, \varphi) \varphi_z)\right),  \quad
  w = \varphi_z E\left( \psi/h - \Pi(\psi/h)\right). 
\end{equation}

\begin{proposition}\label{prop:w}
  Let $w$ be chosen as in \eqref{eq:solution-coh}, then 
  \[
	\cL_z(w) + \cE(\qn, \varphi) \varphi_z =[z \partial_z S]/z - \nabla^+(h \Pi(\psi/h))
  \]
\end{proposition}
\begin{proof}
Since $\psi$ and $w$ satisfies the two equations in \eqref{eq:solution-coh} respectively, direct computation shows that 
\begin{align*}
\cL_z(w)  +	\cE(\qn, \varphi)\varphi_z  &= \nabla^+ (h \nabla(w/\varphi_z)) + \cE(\qn, \varphi)\varphi_z \\
& =  \nabla^+ (h \nabla(E(\psi/h-\Pi(\psi/h)))) + \cE(\qn, \varphi)\varphi_z\\
&= \nabla^+ (h \left(\psi/h-\Pi(\psi/h)  \right)) + \cE(\qn, \varphi)\varphi_z\\
&=\nabla^+  \left(\psi-h\Pi(\psi/h)  \right) + \cE(\qn, \varphi)\varphi_z\\
&=-\cE(\qn, \varphi) \varphi_z + \Pi_+(\cE(\qn, \varphi) \varphi_z)-\nabla^+(h\Pi(\psi/h) +\cE(\qn, \varphi)\varphi_z\\
&= \Pi_+(\cE(\qn, \varphi) \varphi_z)-\nabla^+(h\Pi(\psi/h)\\
&=[z \partial_z S]/z - \nabla^+(h \Pi(\psi/h))
\end{align*}
\end{proof}
In the spirit of proposition \ref{prop:w}, we define the following new error terms: 
\begin{align}
    R_3 &= [z \partial_z S]/z - \nabla^+(h \Pi(\psi/h))\\
    R_4 &= \cE(\qn, \varphi) + \partial_\varphi \cE(\qn, \varphi)(w) 
\end{align}
Note that 
\begin{align*}
    \varphi_{z}R_4 &= \varphi_{z}\left(\cE(\qn, \varphi) +\partial_\varphi \cE(\qn, \varphi)(w) \right) = \varphi_z \cE(\qn, \varphi) + \partial_z \cE\cdot w + \nabla^{+}(h\nabla(w/\phi_z))\\
    &= \partial_z \cE \cdot w + R_3,
\end{align*}
in other words, the error terms $R_3$ and $R_4$ are related by the equation
 \[
R_4 = \partial_z \mathcal{E}\cdot(w/\varphi_z) + R_3/\varphi_z.
\]
While \eqref{eq:solution-coh} solves \eqref{eq:lin-eq}, we broke the symmetry of the system in that $w \circ I$ may not be equal to $w$. This can be dealt with by setting
\begin{equation}  \label{eq:symmetrize}
\Delta \varphi = \frac12(w + w \circ I).	
\end{equation}
\begin{proposition}\label{prop:dphi}
  Let $\Delta \varphi$ be given by \eqref{eq:symmetrize} and \eqref{eq:solution-coh}. Then
  \[
	\cE(\qn, \varphi) + \partial_\varphi \cE(\qn, \varphi)(\Delta \varphi)
	=  \frac12(R_4(\qn, \varphi) + R_4(\qn, \varphi) \circ I). 
  \]
\end{proposition}
\begin{proof}
  Let 
  \[
	\cL_{\bar{z}}(w) = \nabla (\partial_{12} S^+ \varphi_{\bar{z}} \varphi_{\bar{z}}^+ \nabla^+(w/ \varphi_{\bar{z}})), 
  \]
  we check using Lemmas \ref{lem:phi-symmetry} and \ref{lem:ES-symmetry} that 
  \[
	\cL_z(w \circ I) = \cL_{\bar{z}}(w) \circ I.
  \]
  Indeed, since $\varphi$ is symmetric, i.e., $\varphi\circ I = \varphi$, one has that 
  \begin{align*}
     \mathcal{L}_{\bar{z}}(w)\circ I &=  \nabla (\partial_{12} S^+ \varphi_{\bar{z}} \varphi_{\bar{z}}^+ \nabla^+(w/ \varphi_{\bar{z}})) \circ I \\
     &=-\lambda^{-1}\nabla^{+}\left((\partial_{12}S^{+}\circ I)\cdot( \varphi_{\bar{z}}\circ I)\cdot( \varphi^{+}_{\bar{z}}\circ I)\cdot( \nabla^{+}(\frac{w}{\varphi_{\bar{z}}})\circ I)  \right)\\
     &=-\lambda^{-1}\nabla^{+}\left(\partial_{12}S\cdot  (\varphi \circ I)_{z}\cdot( \varphi\circ I)^{-}_{z}\cdot(-\lambda) ( \nabla(\frac{w}{\varphi_{\bar{z}}}\circ I))  \right)\\
     &=\nabla^{+}\left(\partial_{12}S\cdot  \varphi_{z} \varphi^{-}_{z}\cdot  \nabla\left(\frac{w\circ I}{\varphi_{\bar{z}}\circ I}\right)  \right)\\
     &=\nabla^{+}\left(\partial_{12}S\cdot  \varphi_{z} \varphi^{-}_{z}\cdot  \nabla\left(\frac{w\circ I}{(\varphi\circ I)_{z}}\right)  \right)\\
     &=\nabla^{+}\left(\partial_{12}S\cdot  \varphi_{z} \varphi^{-}_{z}\cdot  \nabla\left(\frac{w\circ I}{\varphi_{z}}\right)  \right)\\
     &=\mathcal{L}_z(w\circ I)
  \end{align*}
  From the definition of $\partial_\varphi \cE$,
\begin{align*}
       \partial_\varphi \cE(w)
  &  = \partial_{12} S (w^-) + \partial_{22} S (w)
  + \partial_{11} S^+ (w) + \partial_{12} S^+ (w^+),
\end{align*}
Composing with $I$ on both sides and using the symmetry $S(t_1,t_2) = S(t_2,t_1)$ leads to
\begin{align*}
    ( \partial_\varphi \cE(w))\circ I & = (\partial_{12} S\circ I) (w^-\circ I) + (\partial_{22} S\circ I)
     (w\circ I)\\
  &\quad + (\partial_{11} S^+ \circ I) (w\circ I) + (\partial_{12} S^+ \circ I) (w^+\circ I ) \\
  &=\partial_{12}S^{+}((w\circ I)^{+}) + \partial_{11}S^{+}(w\circ I) + \partial_{22}S(w\circ I) +\partial_{12}S((w\circ I )^{-}) \\
  &=\partial_{\varphi}\mathcal{E}(w\circ I ),
\end{align*}
i.e., 
  \[
	(\partial_\varphi \cE(w)) \circ I = \partial_\varphi \cE(w \circ I). 
  \]
From Proposition \ref{prop:w}, we have
\begin{equation}\label{eq:R_4_without_I}
	\cE + \partial_\varphi \cE (w \circ I) \circ I
	= \cE + \partial_\varphi \cE (w)
	= R_4(\qn, \varphi)
\end{equation}
  and after composing with $I$:
\begin{equation}\label{eq:R_4_with_I}
  	\cE + \partial_\varphi \cE(w \circ I) = R_4 \circ I.    
\end{equation}
The equality desired is then achieved by adding  \eqref{eq:R_4_without_I} and \eqref{eq:R_4_with_I} together.
\end{proof}
Now suppose that $\varphi = \varphi \circ I$ and $\Delta \varphi$ are given by \eqref{eq:solution-coh} and \eqref{eq:symmetrize} respectively.
Finally, let us compute the function $\cE(\qn, \varphi + \Delta \varphi)$. Suppose $\varphi = \varphi \circ I$, and let $\Delta \varphi$ be determined by \eqref{eq:solution-coh} and \eqref{eq:symmetrize}. We can define the following error term:
\begin{align*}
    	R_5 &= \cE(\qn, \varphi + \Delta \varphi)- \frac12(R_4 + R_4 \circ I) \\
	&=\cE(\qn, \varphi + \Delta \varphi) - \cE(\qn, \varphi) -  \partial_\varphi \cE(\qn, \varphi)(\Delta \varphi),
\end{align*}
which represents the higher order error after linear approximation of $\cE(\qn, \varphi + \Delta \varphi)$ based on a pre-fixed $\Delta q$ (hence $q^{*}$)

\section{The Iterative Step}

We summarize the iterative step as follows.

Given $q = 1 + \sum_{k = 1}^\infty q_{2k} t^{2k}$ , $\varphi = \sum_{n = 0}^\infty \sum_{j + k= 2n + 1} \varphi_{j, k} z^k \bar{z}^k$ satisfying $\varphi \circ I = \varphi$, and a parameter $M \in \N$, we perform the iterative step as follows.
\begin{enumerate}
  \item Let $\Delta q = \sum_{k = 2}^{2M} \eta_{2k} t^{2k}$ be defined using Corollary \ref{cor:Delta-q}, i.e.
	\[
	  \Lambda_{2M}\left( [S_{q + \Delta q}] - 1\right) = 0. 
	\]
  \item Denote $\qn = q + \Delta q$,
	\[
  \psi = E^+\Bigl( \cE(q^*, \varphi) \varphi_z - \Pi_+(\cE(q^*, \varphi) \varphi_z) \Bigr),
	\]
	\[
	  h = \partial_{12}S_{\qn}(\varphi^-, \varphi) \varphi_z \varphi_z^-, 
	\]
  \[
 w = \varphi_z E\left( \psi/h - \Pi(\psi/h)\right), 
  \]
  \[
	\Delta \varphi = \frac12(w + w \circ I).
  \]
\item According to the proposition \ref{prop:dphi}, we have that 
  \[
	\cE(q + \Delta q, \varphi + \Delta \varphi) = \frac12(R_4 + R_4 \circ I) + R_5, 
  \]
  where 
  \[
	R_4 = \partial_z \cE(w/\varphi_z) + R_3, 
  \]
  \[
	R_3 = [z \partial_z S]/z - \nabla^+(h \Pi(\psi/h)), 
  \]
  from Proposition \ref{prop:S-2nd}, if we set
  \[
	\kappa = \partial_{12}S (\lambda^{-1}  \varphi_z^- \varphi_{\bar{z}} - \lambda \varphi_{\bar{z}}^- \varphi_z), \quad
	g = \varphi_{\bar{z}}/\varphi_z, 
  \]
  then 
\[
  \Pi\left( \psi/h\right) 
  = R_1 + R_2, 
\]
where 
\[
  R_1 = -\frac{ [\bar{z} \partial_{\bar{z}} S] + [\bar{z} g [z \partial_z S]/ z]}{ \lambda [\kappa]}, \quad 
  R_2 =  \frac{1}{z} \left[\bar{z} \psi(\lambda^{-1} g - \lambda g^-) \frac{[\kappa] - \kappa}{\kappa [\kappa]}\right],
\]
and
  \[
  \kappa - [\kappa] = \tilde{E} \left( \partial_z \cE \cdot \varphi_{\bar{z}} - \partial_{\bar{z}} \cE \cdot \varphi_z \right).
  \]
\end{enumerate}

We now introduce the function spaces for our iteration.

Consider the spaces of power series 
\[
  X_1 = \left\{ \sum_{k = 0}^\infty q_{k} t^{k}\st  q_k \in \C\right\}, \quad
  X_2 = \left\{ \sum_{j, k \ge 0} f_{j, k} z^j \bar{z}^k\st f_{j, k} \in \C\right\}
\]
equipped with the weighted norms
\[
  \|q\|_\rho = \sum_{k = 0}^\infty |q_{k}| \rho^{k}, \quad
  \|f\|_{\rho} = \sum_{j, k \ge 0} |f_{j, k}| \rho^{j + k}.
\]

Denote
\[
  M(q)(t) = \sum_{k \ge 0} |q_k|t^k, \quad
  M(f)(z, \bar{z}) =  \sum_{j, k \ge 0} |f_{j, k}| z^k \bar{z}^k, 
\]
Denote $B_\rho = \{z, \bar{z} \st |z|, |\bar{z}| \le \rho\}$, then if $\|q\|_\rho, \|f\|_\rho < \infty$, $M(q), M(f)$ are convergent series on $|t| < \rho$ and $|z|, |\bar{z}| < \rho$ respectively as complex variable functions. In this case
\[
  \|q\|_\rho = \sup_{|t| < \rho} |M(q)(t)|, \quad
  \|f\|_\rho = \sup_{|z|, |\bar{z}| < \rho} |M(f)(z, \bar{z})|.
\]

Given two power series $f, g \in X_2$, we say $g$ is a majorant series of $f$ (denoted by $f \prec g$) if $|f_{j, k}| \le g_{j, k}$ (in particular, $g_{j, k} \ge 0$). It's clear that we always have $f \prec M(f)$. 

 For either $f \in X_1$ or $X_2$, 
\[
  \|f\|_{\rho, l} = \max_{0 \le |\alpha| \le l} \|\partial^\alpha f\|_\rho
\]
where $\partial^\alpha$ denote any regular or partial derivative of order $|\alpha|$. Also note that for $\tau>0, \sigma>0$
\[
\sup_{x>0} x^{\tau} e^{-\sigma x} = x^{\tau}e^{-\sigma x}|_{x= \tau\sigma^{-1}}
= (\tau\sigma^{-1})^{\tau} e^{-\tau} = \left(\frac{\tau}{\sigma e}\right)^{\tau}
\]

We have the following basic properties for the norm. 
\begin{lemma}\label{lem:norm-properties}
  Let $f, g \in X_2$ and $q \in X_1$, $\rho, \rho_1, \rho_2 > 0$ and $\gamma \in (0, 1)$. 
\begin{enumerate}
 \item If $f = \sum_{j, k} f_{j, k} z^j \bar{z}^k$, then
 \begin{itemize}
     \item  $|f_{j, k}| \le \rho^{-j-k} \|f\|_\rho$. If $f = zg$, then $\|f\|_{\rho} = \rho \|g\|_{\rho}$. 
     \item $\|f\|_{\rho} = \|f^{-}\|_{\rho} = \|f^{+}\|_{\rho}$
 \end{itemize}
 \item $\|fg\|_\rho \le \|f\|_\rho \|g\|_\rho$. 
 \item If $\|f\|_{\rho_1} \le \rho_2$, then $\|q \circ f\|_{\rho_1} \le \|q\|_{\rho_2}$.
 \item We have $\|\partial^\alpha f\|_{\gamma \rho} \le C \rho^{-|\alpha|} (1 - \gamma)^{-|\alpha|} \|f\|_\rho$, where $\alpha = (\alpha_1, \alpha_2)$ is a multi-index and $|\alpha| = \alpha_1 + \alpha_2$ and $C$ is a constant that depends only on $\alpha$.

   The same applies for $q \in X_1$ and $\partial^\alpha$ is a regular derivative.
 \item If $\Lambda_{M-1} f = 0$, then $\|f\|_{\gamma \rho} \le \gamma^M \|f\|_\rho$. 
 \item Consider the operators $\Pi,\Pi^{+}, \nabla, \nabla^{+}$, then 
 \begin{itemize}
     \item    $\|\Pi(f)\|_\rho, \,  \|\Pi_+(f)\|_\rho, \, \|f - \Pi(f)\|_\rho, \, \|f - \Pi_+(f)\|_\rho, \, \|[f]\|_\rho \le \|f\|_\rho$,
      \item $\|\nabla(f)\|_\rho, \|\nabla^+(f)\|_\rho \le 2 \|f\|_\rho$.
 \end{itemize}

 \item Let the operator $A$ be one of $E$, $E^+$ (see \eqref{eq:E}), or $\tilde{E}$ (see \eqref{eq:tilde-E}). There exists $C > 0$ depending only on $\tau$ such that
   \[
	 \|A(f)\|_{\gamma \rho} \le C_\tau (\log(1/\gamma))^{-\tau} \|f\|_\rho
   \]
   where $\tau$ is the Diophantine exponent (see \eqref{eq:diophantine}).
 \item Let $D$, $\bar{D}$ be the operators defined in \eqref{eq:D} and \eqref{eq:D-bar}. Then for $f = f_0 +  \sum_{n \ge 1} f_n (z \bar{z})^n$,
   \[
	 \|D(f)\|_{\gamma \rho} \le \frac{1}{2e} (\log(1/\gamma))^{-1} \|f - f_0\|_\rho, \quad
	 \|\bar{D}(f - f_0)\|_\rho \le \|f\|_\rho. 
   \]
\end{enumerate}	
\end{lemma}
\begin{proof}
All claims follow from definitions.
\begin{enumerate}
\item  Given $\displaystyle f=\sum_{j,k}f_{j,k}z^{j}\bar{z}^{k}$, then $\|f\|_{\rho} \geq |f_{j,k}|\rho^{j+k}$, i.e. $\displaystyle |f_{j,k}|\leq \frac{\|f\|_{\rho}}{\rho}$. If $f=zg$, then 
\[
\|f\|_{\rho} = \|zg\|_{\rho}= \sum_{k,j}|g_{j,k}|\rho^{j+k+1} =\rho \sum_{k,j}|g_{j,k}|\rho^{j+k} = \rho \|g\|_{\rho} 
\]
Finally, by definition of $f^{+}$ and $f^{-}$
\begin{align*}
   \|f^{+}\| &= \sum_{j,k}|f_{j,k}\lambda^{j-k}|\rho^{j+k} = \sum_{j,k}|f_{j,k}|\rho^{j+k} = \|f\|;\\
   \|f^{-}\| &= \sum_{j,k}|f_{j,k}\lambda^{k-j}|\rho^{j+k} = \sum_{j,k}|f_{j,k}|\rho^{j+k} = \|f\|.
\end{align*}

\item Note that as functions, $M(fg) \prec M(f) M(g)$. We have
\[
  \|fg\|_\rho = \sup_{B_\rho} M(fg) 
  \le \sup_{B_\rho} M(f)M(g) 
  = \|f\|_\rho \|g\|_\rho.
\]
\item Observe that $M(f \circ g) \prec M(f) \circ M(g)$. We abuse notation slightly and use $B_\rho$ to also denote $\{t \st |t| \le \rho\}$. Then
\[
  \|q \circ f\|_{\rho_1} = \sup_{B_{\rho_1}} |M(q \circ f)| \le \sup_{B_{\rho_1}} |M(q) \circ M(f)| \le \sup_{B_{\rho_2}} |M(q)| 
\]
since $\sup_{B_{\rho_1}} |M(f)| = \|f\|_{\rho_1} \le \rho_2$.
\item Observe that $M(\partial^\alpha f) = \partial^{\alpha}M(f)$ as formal power series. If $\|f\|_\rho < \infty$, $M(f)$ is holomorphic on $|z|, |\bar{z}| < \rho$ (as a function of two complex variables). Apply the Cauchy estimate, we have
\[
  \|\partial^\alpha f\|_{\gamma \rho} = \sup_{B_{\gamma \rho}} |\partial^\alpha M(f)|
  \le C_\alpha ((1- \gamma) \rho)^{-|\alpha|} \sup_{B_\rho} |M(f)| 
\]
for a constant $C_\alpha$. The same proof applies to $q \in X_1$ and $\partial^\alpha$ is a regular derivative.
\item One has that
\[
  \|f\|_{\gamma \rho} = \sum_{j + k \ge M} |f_{j, k}| \gamma^{j + k} \rho^{j + k}
  \le \gamma^{M} \sum_{j + k \ge M} |f_{j, k}| \rho^{j + k} \le \gamma^M \|f\|_\rho. 
\]
\item The first formula in this statement follows from the definition of $\|\cdot \|_{\rho}$ directly. Indeed, 
\begin{align*}
    \|f-\Pi^{+}(f)\|_{\rho} + \|\Pi^{+}(f)\|_{\rho} &=  \sum_{j}^{\infty}|f_{j,j+1}|\rho^{2j+1} + \sum_{\substack{j,k\\k\neq j+1}  } |f_{j, k}| \rho^{j + k} = \|f\|_{\rho} \\
     \|f-\Pi(f)\|_{\rho} + \|\Pi(f)\|_{\rho} &=  \sum_{j}^{\infty}|f_{j+1,j}|\rho^{2j+1} + \sum_{\substack{j,k\\k\neq j-1}  } |f_{j, k}| \rho^{j + k} = \|f\|_{\rho}\\
       \|[f]\|_{\rho}  &=\sum_{j}^{\infty}|f|_{jj} \rho^{2j}
\le \sum_{j,k} |f_{j, k}| \rho^{j + k} =\|f\|_{\rho},  
\end{align*}
from the above the desired inequalities follows directly. Concerning the second formula of this statement, by definition, 
\[
\|\nabla^{+} f\|_{\rho}= \sum_{j,k} |f_{jk}(1-\lambda^{j-k+1})|\rho^{j+k} \leq  \sum_{j,k} 2|f_{jk}|\rho^{j+k} = 2\|f\|_{\rho},
\]
similarly $\|\nabla f\|_{\rho} \le 2\|f\|_{\rho}$.
\item  Recall that $\lambda$ satisfies the Diophatine condition, i.e.,   $\forall l\neq0, |\lambda^l - 1| \ge c |l|^{-\tau},$ As a result,
\[
  |(A(f))_{j, k}|  \le  \frac{1}{c}|f_{j, k}| (j + k)^\tau, \forall j+k\neq 0
\]
for $A = E, E^+, \tilde{E}$. By the Diophatine property,
\begin{align*}
      \|A(f)\|_{\gamma \rho} &\le |f_{0,0}|+ \frac{1}{c}\sum_{\substack{j, k\\ j+k>0}} |f_{j, k}| (j + k)^\tau \gamma^{j + k} \rho^{j + k} \le \frac{1}{c}\sum_{j, k } |f_{j, k}| (j + k)^\tau \gamma^{j + k} \rho^{j + k}\\
      &\le \frac{1}{c}\|f\|_\rho \sup_n n^\tau \gamma^n, 
\end{align*}
where we have used that $0<c<1$.
Write $\sigma = \log(1/\gamma)$, a direct calculation shows that $\displaystyle \sup_{x > 0} x^\tau e^{-\sigma x} \le \left(\frac{\tau}{\sigma e}\right)^{\tau}$. Denote 
\[
C_{c,\tau}=\frac{1}{c}\left(\frac{\tau}{e}\right)^\tau, 
\]
the above inequality becomes 
\[
 \|A(f)\|_{\gamma \rho} \leq C_{c,\tau} \left(\log\frac{1}{\gamma}\right)^{-\tau}\|f\|_{\rho}
\]
\item One has that
\[
  \|D(f)\|_\rho = \sum_{k \ge 0} (k) f_{k, k}\gamma^{2k} \rho^{2k}
  \le \frac12 \|f-f_0\|_\rho \sup_n n \gamma^n = \frac{1}{2e} (\log(1/\gamma))^{-1} \|f-f_0\|_\rho,
\]
and
\[
  \|\bar{D}(f)\|_\rho = \sum_{k \ge 1} \frac{|f_{k, k}|}{k} \rho^{2k} \le \sum_{k \ge 1} |f_{k, k}| \rho^{2k} = \|f-f_0\|_\rho\leq \|f\|_{\rho}.
\]
\end{enumerate}

\end{proof}

We start with some basic norm estimates for the quantities involved in the iteration. 

\begin{lemma}\label{lem:basic-est}
Let $\mu = |1 + \lambda|$. Given $C_0 > 0$, there exists $r_0 > 0$ such that the following hold. Let $\rho_0 \in (0, r_0)$, and suppose
\[
  \|\varphi\|_{\rho_0, 3} \le C_0, \quad \|q\|_{\rho_0, 3} \le C_0, 
\]
then there exists $C_1, C_2, C_3> 0$ such that the following hold for any $\rho \in (0, \rho_0)$ and $\gamma \in (0, 1)$. 
\begin{enumerate}
  \item  $\|\xi\|_\rho \le \mu\rho(1 + C_1 \rho_2^2)$.

  \item Recall that $p(t) = \cos(t)$. We have $\|p(\zeta)\|_\rho \le C_2$.

  \item $\|\varphi_z - 1\|_\rho \le C_1 \rho^2$, and $\|\varphi_z\|_\rho, \|1/\varphi_z\|_\rho \le C_1$. The same estimates hold when $z$ is replaced with $\bar{z}$. 

  \item $\|[z\partial_z S_{\qn}]/z\|_{\gamma\rho} = \|D([S_{\qn}])/z\|_{\gamma \rho} \le \frac{C_1}{\gamma \rho \log(1/\gamma)} \|[S_{\qn}] - 1\|_\rho$. The same holds for $z$ replaced with $\bar{z}$.

  \item $\|\Pi_+(\cE(\qn, \varphi))\varphi_z\|_{\gamma\rho} = \|[z \partial_z S_{\qn}]/z\|_{\gamma\rho} \le \frac{C_1}{\gamma \rho \log(1/\gamma)} \|[S_{\qn}] - 1\|_\rho$.  

\item Let $\psi$ be as defined in \eqref{eq:solution-coh}, then
  \[
\|\psi\|_{\gamma \rho} \le \frac{C_2}{(\log(1/\gamma))^\tau} \|\cE(\qn, \varphi)\|_\rho  \], where $\tau$ is the Diophantine exponent of $\lambda$ (see \eqref{eq:diophantine}).
\item Let $s_0 = \frac14(1 + 2q_2)$, which is the constant term of $\partial_{12} S(t_1, t_2)$. Then 
\[\|\partial_{12} S_{\qn}(\varphi^-, \varphi) - s_0\|_\rho \le C_2 \mu \rho(1 + C_1 \rho^2)\|\qn\|_{\mu \rho(1 + C_1 \rho^2), 3}.
\]
Suppose that
  \begin{equation}  \label{eq:small-q}
  	\|\partial_{12} S_{\qn}(\varphi^-, \varphi) -  s_0\|_\rho  < \frac12|s_0|, 
  \end{equation}
  then $\|h\|_\rho, \|1/h\|_\rho \le C_3$, where $h = \partial_{12}S_{\qn}(\varphi^-, \varphi) \varphi_z \varphi_z^-$.
\item Let $w$ be defined in \eqref{eq:solution-coh}, then 
  \[
	\|w\|_{\gamma \rho} \le \frac{C_2}{(1 - \gamma)^\tau} \|\psi\|_\rho \|1/h\|_\rho. 
  \]
\item The next few items estimate the remainder in Proposition \ref{prop:small-average-psi}. 

  Let $\kappa$ be as defined in \eqref{eq:kappa}. Then similar to the estimate for $h$, we have that $\|\kappa - (\lambda^{-1}-\lambda)s_0\|_{\rho} \le C_2 \|\qn\|_{\mu \rho(1 + C_1 \rho^2), 3}$. If \eqref{eq:small-q} holds, then
  \[
	\|\kappa\|_\rho, \|[\kappa]\|_\rho, \|1/\kappa\|_\rho, \|1/[\kappa]\|_\rho \le C_3.
  \]

\item 
  \[
	\|\kappa - [\kappa]\|_{\gamma^2 \rho} \le
\frac{C_2}{(\log(1/\gamma))^\tau (1 - \gamma) \rho} \|\cE\|_\rho.
  \]
\item Let $R_1, R_2$ be as defined in Proposition \ref{prop:S-2nd}. Then
  \[
	\|R_1\|_{\gamma \rho} \le \frac{C_2}{\gamma \rho \log(1/\gamma)} \|1/[\kappa]\|_{\gamma \rho} \|[S_{\qn}] - 1\|_\rho. 
  \]
  \[
	\|R_2\|_{\gamma \rho} \le \frac{C_2}{\gamma \rho}\|1/\kappa\|_{\gamma \rho} \|1/[\kappa]\|_{\gamma \rho} \|\psi\|_{\gamma \rho} \|\kappa - [\kappa]\|_{\gamma \rho}. 
  \]
\item Recall that from Proposition \ref{prop:S-2nd}
  \[
	\Pi(\psi/h) = R_1 + R_2. 
  \]
Let $R_3, R_4$ be as defined in Proposition \ref{prop:w}. Then 
  \[
	\|R_3\|_{\gamma \rho} \le \frac{C_2}{\gamma\rho(1 - \gamma)} \|[S_{\qn}] - 1\|_\rho
	+ C_2 \|h\|_{\gamma \rho} \|R_1 + R_2\|_{\gamma \rho},
  \]
  \[
	\|R_4\|_{\gamma \rho} \le \frac{C_2}{1 - \gamma} \|\cE(\qn, \varphi)\|_\rho \|w\|_\rho + C_2 \|R_3\|_{\gamma \rho}. 
  \]
\item Suppose for some $\rho_1 \in (0, \rho)$, we have
\begin{equation}  \label{eq:delta-phi-assumption}
	\|\Delta \varphi\|_{\rho_1, 3} < C_0, 
\end{equation}
  then
  \[
	\|R_5\|_\rho \le C_2 \|\qn\|_{\mu \rho_1( 1+ C_1 \rho_1^2), 3} \|\Delta \varphi\|_{\rho_1}^2. 
  \]
\end{enumerate}
\end{lemma}
\begin{proof}
  Within this proof, $f \lesssim g$ stands for $f \le C g$ for a constant $C$ depending only on $C_0$ and the constant $c$ in the Diophantine condition \eqref{eq:diophantine}.
\begin{enumerate}
    \item Since $\varphi = z + \bar{z} + O_3$, we have $\xi = \frac12 (\lambda^{-1} + 1) z + \frac12 (\lambda + 1) z + O_3 = \xi_0 + O_3$, hence
  \[
	\|\xi - \xi_0\|_\rho = \sup_{B_\rho} |M(\xi - \xi_0)|
	\le \rho^3 \sup_{B_\rho} \sup_{|\alpha| \le 3} |\partial^\alpha M(\xi - \xi_0)|\le \|\varphi\|_{\rho,3}.
  \]
  We will set 
\begin{equation}  \label{eq:C1}
C_1 = 2 C_0.	
\end{equation}
Note that $\|\xi_0\|_\rho = \mu \rho$, we have
\[
  \|\xi\|_\rho \le \|\xi_0\|_\rho + \|\xi - \xi_0\|_\rho 
  \le \mu \rho + \frac12 C_1 \rho^3 = \mu \rho(1 + \frac12 C_1 \rho^2).
 \]
 This estimate is \emph{better} than what's required. In fact, the choice in \eqref{eq:C1} is to ensure the estimate holds if we replace $C_0$ with $2C_2$. This is needed in item 13. 
 \item Since $p$ is an entire function, we can choose say $C_2 = \|p\|_1$.
 \item We have
  \[
	\|\varphi_z - 1\|_\rho = \sup_{B_\rho} |M(\varphi_z - 1)|
	\lesssim \rho^2 \sup_{B_\rho} \sup_{|\alpha| \le 3} |\partial^\alpha M(\varphi)| 
	\lesssim \rho^2
  \]
  by the same argument as item 1. This immediately implies $\|\varphi_z\|_\rho \lesssim 1$. By choosing $\rho_0$ small enough, we have $\|\varphi_z - 1\|_\rho < \frac12$. Since $\displaystyle \frac{1}{\varphi_z} = \sum_{k=0}^{\infty}(1-\varphi_z)^k$, it follows that  $\|1/\varphi_z\|_\rho < 2$. 
\item By item 1\& item 8 in Lemma~\ref{lem:norm-properties}
  \[
	\|D([S_{\qn}])/z\|_{\gamma \rho} \le (\gamma \rho)^{-1} \|D([S_{\qn}])\|_{\gamma \rho} \lesssim \frac{1}{\gamma \rho \log(1/\gamma)} \|[S_{\qn}] - 1\|_\rho,
  \]
\item This is a direct consequence of \eqref{eq:S-to-E2}.
\item By definition 
  \[
	\psi = E^+ \left( \cE(\qn, \varphi) \varphi_z - \Pi_+ (\cE(\qn, \varphi)\varphi_z)\right). 
  \]
  By item 6 and 7 of Lemma~\ref{lem:norm-properties},
  \[
	\|\psi\|_{\gamma \rho} \lesssim \frac{1}{(\log(1/\gamma))^\tau} \|(\id - \Pi_+) \cE\varphi_z\|_\rho 
	\le \frac{1}{(\log(1/\gamma))^\tau} \|\cE\varphi_z\|_\rho
	\lesssim \frac{1}{(\log(1/\gamma))^\tau} \|\cE\|_\rho
  \]
  where in the last inequality, we used $\|\varphi_z\|_\rho \lesssim 1$.
 \item A direct calculation shows that
  \[
	\partial_{12} S(t_1, t_2) = \frac14 \cos\frac{t_2 - t_1}{2} \cdot (q + q'')\left( \frac{t_1 + t_2}{2}\right).
  \]
  This also implies
  \[
	s_0 = \frac14(1 + 2q_2).
  \]
 Let $\rho_1 = \mu\rho(1 + C_1 \rho^2)$. Recall that $\|\eta\|_{\rho}\leq \rho_1$. By lemma \ref{lem:norm-properties}, items 1, 2 and 3, one has that 
 \begin{align}\label{eq:S12-split}
   \|\partial_{12} S(\varphi^-, \varphi) - s_0\|_{\rho} 
	\leq & \| \frac14(q'' + q-s_0)(\xi)\|_{\rho} + \|\frac14 (q'' + q)(\xi) (\cos(\zeta) - 1)\|_{\rho}\notag \\
	\leq & \|\frac14(q'' + q-s_0)\|_{\rho_1} + \|\frac14 (q'' + q)(\xi)\|_{\rho} \|(\cos(\zeta) - 1)\|_{\rho} \notag \\
	\leq & \|\frac14(q'' + q-s_0)\|_{\rho_1} + \|\frac14 (q'' + q)\|_{\rho_1} \|\zeta\|^2_{\rho}\notag\\
	\le& \rho_1\|q\|_{\rho_1,3} + \rho^2\|q\|_{\rho_1,3}\notag \\
	\lesssim& \rho_1 \|q\|_{\rho_1,3}
\end{align} 
where the last step is due to that $\rho^2 < \rho_1$ for for $\rho_0$ small enough.
  To prove the second part of this item, note that when $\rho$ is small, estimate \eqref{eq:S12-split} implies \eqref{eq:small-q} 
  \[
	\|\partial_{12}S(\varphi^-, \varphi) - s_0\|_{\rho} < \frac12 |s_0|. 
  \]
  Since $\|\varphi_z - 1\|_\rho, \|\varphi_z^- - 1\|_\rho \lesssim \rho^2$, if $\rho_0$ is small enough, we can ensure
  \begin{align*}
     \|h - s_0\|_\rho = \| \partial_{12}S(\varphi^-, \varphi) \varphi_z \varphi_z^- - s_0\|_\rho< \frac34 |s_0|.
  \end{align*}
It follows that 
\begin{align*}
 \|h\|_\rho \leq \|h-s_0\|_{\rho} + \|s_0\|_\rho \leq \frac74|s_0|\\
 \|\frac1h\|_\rho = \|s_0^{-1}\sum_{k = 0}^\infty \left( \frac{s_0 - h}{s_0}\right)^k\|_{\rho}\leq \frac4{|s_0|}
\end{align*}

  \item By lemma \ref{lem:norm-properties} item 2, item 6 and item 8,  We have
\begin{align*}
  	\|\varphi_z E \circ (\id - \Pi) (\psi/h)\|_{\gamma \rho}
	&\lesssim \|E \circ (\id - \Pi) (\psi/h)\|_{\gamma \rho}\\
	&\lesssim \frac{1}{(\log(1/\gamma))^\tau}\|E \circ (\id - \Pi) (\psi/h)\|_{\rho}\\
	&\lesssim \frac{1}{(\log(1/\gamma))^\tau}\|(\id - \Pi) (\psi/h)\|_{\rho}\\
	&\lesssim \frac{1}{(\log(1/\gamma))^\tau}\|(\psi/h)\|_{\rho}\\
	&\lesssim \frac{1}{(\log(1/\gamma))^\tau} \|\psi\|_\rho\|1/h\|_\rho.    
\end{align*}
\item Recall that 
  \[
	\kappa = \partial_{12} S \left(\lambda^{-1} \varphi_z^- \varphi_{\bar{z}}
   - \lambda \varphi_{\bar{z}}^- \varphi_z \right).
  \]
Denote $g = \lambda^{-1} \varphi_z^- \varphi_{\bar{z}} - \lambda \varphi_{\bar{z}}^- \varphi_z $. Since $\varphi(z, \bar{z}) = z + \bar{z} + O_3$, it follows that $\varphi_z^- \varphi_{\bar{z}}= 1+ O_2$, $\varphi_{\bar{z}}^- \varphi_z=1+ O_2$,    we have thus 
\[
\|g - (\lambda^{-1} - \lambda)\|_\rho  =\|\lambda^{-1}(\varphi_z^- \varphi_{\bar{z}}-1) -\lambda(\varphi_{\bar{z}}^- \varphi_z-1)\|_\rho   \lesssim \rho^2.
\]
The estimates for $\|\kappa\|_{\rho}$ and $\|\frac{1}{\kappa}\|_{\rho}$ is identical to the estimate of $h$ in item 7, while those for $\|[\kappa]\|_{\rho}$ and $\displaystyle \|[\frac{1}{\kappa}]\|_{\rho}$ follows from item 6 in Lemma \ref{lem:basic-est}.
\item We have (see \eqref{eq:non-res-kappa})
  \[
  \kappa - [\kappa] = \tilde{E} \left( \partial_z \cE \cdot \varphi_{\bar{z}} + \partial_{\bar{z}} \cE \cdot \varphi_z \right).
  \]
Then by item 4 and item 7 in Lemma \ref{lem:basic-est},
  \[
	\|\kappa - [\kappa]\|_{\gamma^2 \rho}
	\lesssim \frac{1}{(\log(1/\gamma))^\tau} \|\partial_z \cE \varphi_{\bar{z}} + \partial_{\bar{z}} \cE \varphi_z\|_{\gamma \rho}
	\lesssim \frac{1}{(\log(1/\gamma))^\tau (1 - \gamma) \rho} \|\cE\|_\rho.
  \]
\item By definition (see \eqref{eq:r1andr2})
  \[
  R_1 = \frac{- [\bar{z} \partial_{\bar{z}} S] + [\bar{z} g [z \partial_z S]/ z]}{ \lambda z [\kappa]}, \quad 
  R_2 =  \frac{1}{z} \left[ \psi(\lambda^{-1} g - \lambda g^-) \frac{[\kappa] - \kappa}{\kappa [\kappa]}\right].
  \]
 Since $\|g\|_{\rho} \lesssim 1$, by item 4 of Lemma \ref{lem:basic-est}
 \begin{align}
    \|R_1\|_{\gamma\rho} &\leq  \left\|\left(\frac{- [\bar{z} \partial_{\bar{z}} S]}{z} + \frac{[\bar{z} g [z \partial_z S]/ z]}{z} \right) \frac{1} {\lambda[\kappa]}\right\|_{\gamma\rho} \notag\\
    &\leq \left(\| [\bar{z} \partial_{\bar{z}} S]/z\|_{\gamma\rho}+\|g\|_{\gamma \rho}\|[z\partial_zS]/z\|_{\gamma\rho}\right)\|\frac{1}{\kappa}\|_{\gamma \rho}\notag\\
    &\lesssim \frac{1}{\gamma \rho \log(1/\gamma)} \|1/[\kappa]\|_{\gamma \rho} \|[S_{\qn}] - 1\|_\rho. 
 \end{align}

  The estimate for $R_2$ is straight forward:
  \[
	\|R_2\|_{\gamma \rho} \lesssim (\gamma \rho)^{-1} \|\psi\|_{\gamma \rho} \|[\kappa] - \kappa\|_{\gamma \rho} \|1/\kappa\|_{\gamma \rho} \|1/[\kappa]\|_{\gamma \rho}.
  \]
  \item  We have
  \[
	R_3 = [z \partial_z S]/z - \nabla^+(h \Pi(\psi/h)), \quad
	R_4 = \partial_z \cE (w/\varphi_z) + R_3/\varphi_z.
  \]
  Using item 6 of Lemma~\ref{lem:norm-properties} and item 4 of this lemma, we have
  \begin{align*}
      	\|R_3\|_{\gamma \rho} &\leq \|[z\partial_zS]/z\|_{\gamma\rho} + \|h\|_{\gamma \rho }\|\Pi(\psi/h)\|_{\gamma \rho}\\
	&\leq \|[z\partial_zS]/z\|_{\gamma\rho} + \|h\|_{\gamma \rho }\|\psi/h\|_{\gamma \rho}\\
	&\lesssim \frac{1}{\gamma \rho \log(1/\gamma)} \|[S_{\qn}] - 1\|_\rho + \|h\|_{\gamma \rho} \|R_1 + R_2\|_{\gamma \rho}.\\
	\|R_4\|_{\gamma \rho} &\leq\left( \|\partial_{z}\cE(w)\|_{\gamma \rho}  + \|R_3\|_{\gamma \rho}\right) \|\frac{1}{\varphi_z}\|_{\gamma\rho}.\\
	&\lesssim  \frac{1}{(1 - \gamma)\rho} \|\cE\|_\rho \|w\|_{\gamma \rho} + \|R_3\|_{\gamma \rho}\|
  \end{align*}
\item We have
  \[
	R_5 = \int_0^1 (1 - t)\frac{d^2}{dt^2}\cE(\qn, \varphi + t \Delta \varphi)dt 
  \]
  as formal power series. Write $\varphi_t = \varphi + t \Delta \varphi$, $\xi_t = (\varphi_t^- + \varphi_t)/2$, $\zeta_t = (\varphi_t - \varphi_t^-)/2$, we have
  \[
	\cE(q, \varphi_t)
	=  \frac12 \left( 
	q'(\xi_t) p(\zeta_t) - q(\xi_t) p'(\zeta_t) + q'(\xi_t^+) p(\zeta_t^+) + q(\xi_t^+)p'(\zeta_t^+)
	\right).
  \]
  Hence $\frac{d^2}{dt^2} \cE(q, \varphi_t)$ is an expression of the type
  \[
	\sum_{j, k, \beta = (\beta_1, \beta_2)} c_{j, k, \beta} q^{(j)}(\xi_t^{\beta_1}) p^{(k)}(\zeta_t^{\beta_2}) V_{j, k, \beta} W_{j, k, \beta} 
  \]
  where $j + k = 3$, $\beta_1, \beta_2$ is either empty or $+$, and $V, W$ takes values from $\Delta \xi, \Delta \xi^+, \Delta \zeta, \Delta \zeta^+$. 

  Assumption \eqref{eq:delta-phi-assumption} implies $\|\varphi_t\|_{\rho_1, 3} \le 2 C_0$, which implies $\|\xi_t\|_{\rho_1}, \|\xi_t^+\|_{\rho_1} \le \mu \rho_1(1 + C_1 \rho_1^2)$ with the same $C_1$ as in Item 1 of this lemma. This is why in \eqref{eq:C1} we chose $C_1$ to be twice as large as necessary. It follows that
  \[
	\|q^{j}(\xi_t^{\beta_1})\|_{\rho_1}  \le \|q\|_{\mu \rho_1(1 + C_1 \rho_1^2), 3}.
  \]
  Since 
$\|p^{(k)}(\zeta_t^{\beta_2})\|_\rho \lesssim 1$, $\|\Delta \zeta\|_\rho, \|\Delta \zeta^+\|_\rho, \|\Delta \xi\|_\rho, \|\Delta \xi^+\|_{\rho_1}\lesssim \|\Delta \varphi\|_{\rho_1}$, we obtain
\[
  \left\|
	 \frac{d^2}{dt^2} \cE(\qn, \varphi_t) 
	 \right\|_{\rho_1}
	 \lesssim \|q\|_{\mu \rho_1(1 + C_1 \rho_1^2), 3} \|\Delta \varphi\|_{\rho_1}^2. 
\]

\end{enumerate}

\end{proof}

The following proposition estimate the norm of the mapping from $[S] - 1$ to $\Delta q$. 

\begin{proposition}\label{prop:bound-Delta-q}
  Let $C_0, C_1, r_0$ be as given in Lemma \ref{lem:basic-est}. There exists $C_3, C_4 > 0$ depending only on $C_0$ such that if $\rho_0 \in (0, r_0)$,
  \[
  \|\varphi\|_{\rho_0, 3} \le C_0, \quad \|q\|_{\rho_0, 3} \le C_0,
  \]
$M \in \N$, and $\rho \in (0, \rho_0)$ satisfies
\begin{equation}\label{eq:M52rho2}
  C_3 M^{\frac52} \rho^2 < 1 ,  
\end{equation}
the polynomial $\Delta q = \sum_{k = 2}^{2M} \eta_{2k}$ as defined in \eqref{eq:Delta-q} satisfies
\[
  \|\Delta q\|_{\rho_2} \le  C_4 \sqrt{M} \, \|\Lambda_{2M} [S_q] - 1\|_{\rho_1}, 
\]
where 
  \[
	\rho_1 = \frac{\rho}{(1 + C_1 \rho^2)^M}, \quad \rho_2 = \mu\rho_1.
  \]
\end{proposition}
The proof is contained in Appendix~\ref{sec:delta-q}.

Lemma \ref{lem:basic-est} and Proposition \ref{prop:bound-Delta-q} will be used to estimate the norm of the solution $(q, \varphi)$ in the iteration. The next lemma shows that the iteration show that the iteration double the number of terms of correct coefficients.

\begin{lemma}\label{lem:order-inc}
Suppose $(q, \varphi)$ are chosen such that 
\[
  \cE(q, \varphi) = O_{2N + 1}
\]
for some $N \in \N$. (Note that if $q$ is an even series and $\varphi$ and odd series, then $\cE(q, \varphi)$ is an odd series). Suppose $\Delta q$ is chosen such that 
\[
  \Lambda_{2M}\left( [S_{q + \Delta q}] - 1\right) = 0
\]
for some $M  \in \N$.

Then for $K = \min\{M, 2N\}$, we have
\[
  \Delta q = O_{2N+2}, \quad \Delta \varphi = O_{2N + 1}, \quad
  \cE(q + \Delta q, \varphi + \Delta \varphi) = O_{2K + 1}. 
\]
\end{lemma}
\begin{proof}
  From $\cE(q, \varphi) = O_{2N + 1}$ and \eqref{eq:S-to-E2}, we get 
  \[
	[S_q] - 1 = O_{2N + 2}. 
  \]
  Therefore the solution $\Delta q$ to $[S_{q + \Delta q}] = 1$ satisfied 
  \[
	\Delta q = O_{2N + 2}. 
  \]
  One checks by definition that 
  \[
	\psi = O_{2N + 1}, \quad
	w = O_{2N + 1}, \quad
	\Delta \varphi = O_{2N + 1}.
  \]
  To compute $\cE(q + \Delta q, \varphi + \Delta \varphi)$, we first apply Proposition \ref{prop:S-2nd} to get 
  \[
	R_1 + R_2 = O_{2K + 1}. 
  \]
  Since $[z\partial_z S_{\qn}]/z = D([S])/z = O_{2M + 1}$, we have
  \[
	R_3 = [z\partial_z S_{\qn}]/z + \nabla^+(h (R_1 + R_2)) = O_{2K + 1}, 
  \]
  \[
	R_4 = \partial_z\cE (w/\varphi_z) + R_3/\varphi_z = O_{4N + 1} + O_{2K + 1} = O_{2K + 1}, 
  \]
  \[
	R_5 = O(\Delta \varphi^2) = O_{4N + 2}. 
  \]
  Combine everything, we get
  \[
	\cE(q + \Delta q, \varphi + \Delta \varphi) = O_{2K + 1}.
  \]
\end{proof}

One consequence of Lemma~\ref{lem:order-inc} is that we have an alternative proof for Treschev's result.

\begin{corollary} \label{cor:treschev-alt}
There exists power series
\[
  q^\infty(t) = 1 + q_2 t^2 + \sum_{k = 2}^\infty q_{2k} t^{2k}, \quad
  \varphi^\infty(z, \bar{z}) =  z + \bar{z} + \sum_{n = 0}^\infty \sum_{j + k = 2n + 1} \varphi_{j, k} z^j \bar{z}^k
\]
solving the equation $\cE(q, \varphi) = 0$ formally.
\end{corollary}
\begin{proof}
  Let $\qu{0} = 1 + q_2 t^2$ and $\phiu{0} = z + \bar{z}$. We have
  \[
	\cE(\qu{0}, \phiu{0}) = O_{2N_0 + 1}, 
  \]
  where $N_0 = 1$. Set $M_0 = 2$, and inductively $M_{n + 1} = 2M_n$, $N_{n + 1} = 2N_n$. Let $\Delta \qu{n}$, $\Delta \phiu{n}$ be chosen using the iterative step, and set $\qu{n+1} = \qu{n} + \Delta \qu{n}$, $\phiu{n+1} = \phiu{n} + \Delta \phiu{n}$. Then
  \[
	\Delta \qu{n} = O_{2N_n + 2}, \quad  \Delta \phiu{n} = O_{N_n + 1}, \quad
	\cE(\qu{n}, \phiu{n}) = O_{4N_n + 1}. 
  \]
  It's clear that $\qu{n} \to q^\infty$, $\phiu{n} \to \phi^\infty$ termwise, and 
  \[
   \cE(q^\infty, \varphi^\infty) = 0
  \]
  formally.
\end{proof}

\begin{remark} \
	\begin{enumerate}
	  \item This proof of Treschev's theorem does not immeditately give uniqueness, so we still refer to \cite{Tre13} for uniqueness. 
	  \item The proof of our main theorem uses the same iterative process, but to allow the norm estimates, we will choose different $\qu{0}$, $\phiu{0}$, $M_n$ and $N_n$.  
	\end{enumerate}
\end{remark}

\section{KAM induction}	

\begin{proposition}[Iterative Step]\label{prop:iterative}
  Let 
  \[
q = 1 + \sum_{k = 1}^\infty q_{2k} t^{2k}, \quad
\varphi = z + \bar{z} + \sum_{n = 1}^\infty \sum_{j + k = 2n + 1} \varphi_{j, k} z^j \bar{z}^k,
  \]
  with $\varphi \circ I = \varphi$. Given $C_0 > 0$, let $C_1$ be as in Lemma~\ref{lem:basic-est}. There exists $\rho_0 > 0$, $\epsilon_0 > 0$, and $C_4, C_5 > 1$, such that the following holds. Suppose $\rho \in (0, \rho_0)$, 
  \[
	\|q\|_{\mu\rho(1 + C_1 \rho^2), 3} \le C_0, \quad  \|\varphi\|_{\rho, 3} \le C_0,
  \]
  \[
	\|\cE(q, \varphi)\|_{\rho} < \epsilon \in (0, \epsilon_0),
  \]
 and that $\gamma_0, \gamma_1, \gamma_2 \in (0, 1)$ and $M \in \N$ verifies the following conditions:
  \begin{enumerate}[(a)]
	\item $C_4 M^{\frac{5}{2}} \rho^2 < 1$
	\item $\gamma_1^{2M} < \epsilon$. 
	\item $(1 + C_1 \rho^2)^{-M} > \gamma_2$. 
	\item $C_4 \mu \rho < \gamma_0 |s_0|$ and $1 + C_1 \rho^2 < 1/\gamma_0$, where $\displaystyle s_0 = \frac14(1 + 2q_2)$.
	\item $\epsilon < (1 - \gamma_0)^3 \gamma_0^3\gamma_2^3 \rho^3$. 
  \end{enumerate}
  Then for $\qn = q + \Delta q$, $\phin = \varphi + \Delta \varphi$ as before, $\bar{\gamma} = \gamma_0^4 \gamma_1 \gamma_2$, and $\rhon = \bar{\gamma} \rho$, there exists $C_\gamma > 0$ depending on $\gamma_0, \gamma_1, \gamma_2, C_0$, and $\tau$, such that 

  we have
 \begin{enumerate}
   \item  $\|\Delta q\|_{\mu \rhon (1 + C_1 (\rhon)^2), 3} \le C_{\gamma} \, \epsilon \rho^{-3}$.
   \item $\|\Delta \varphi\|_{\rhon, 3} \le C_{\gamma} \, \epsilon \rho^{-3}$.
   \item $\|\cE(\qn, \phin)\|_{\rhon} \le  C_{\gamma} \, \epsilon^2 \rho^{-7}$.
 \end{enumerate}
\end{proposition}
\begin{proof}
  We will again use the $f \lesssim g$ notation to denote $f \le Cg$ for $C>0$ depending only on $C_0$ and $c$. Most of the estimates comes from Lemma \ref{lem:basic-est}.

We have
\[
  \|\varphi_z\|_\rho \lesssim 1, 
\]
\[
  \|[S_q] - 1\|_\rho = \|\bar{D}(z \Pi_+(\cE(q, \varphi) \varphi_z))\|_\rho 
  \lesssim \|z \Pi_+(\cE(q, \varphi) \varphi_z))\|_\rho 
  \lesssim \rho \|\cE \varphi_z\|_\rho \lesssim \epsilon \rho.
\]
Set $\rho_1 = \frac{\rho}{(1 + C_1 \rho^2)^M}$ and $\rho_2 = \mu \rho_1$, the conditions of Proposition~\ref{prop:bound-Delta-q} is satisfied since
\[
  C_4 M^{\frac{5}{2}} \rho^2 < 1.
\]
As a result, 
\[
  \|\Delta q\|_{\rho_2} \lesssim \sqrt{M} \rho \cdot \epsilon \le \epsilon, 
\]
noting that the assumption
\[
  C_4 M^{\frac{5}{2}} \rho^2 < 1, \quad C_4 > 1
\]
implies $\sqrt{M} \rho < 1$. Condition (c) implies $\gamma_2 \rho < \rho_1$. Now set 
\[
  \bar{\rho}_1 = \gamma_0 \rho_1, \quad \bar{\rho}_2 = \mu \bar{\rho}_1 (1 + C_1 \bar{\rho}_1^2), 
\]
then $\bar{\rho}_2 \le \mu \rho_1 \gamma_0 (1 + C_1 \rho^2) < \rho_2$ (we used the second part of condition (d)). As a result, $\|\Delta q\|_{\bar{\rho}_2} \le \|\Delta q\|_{\rho_2} \lesssim \epsilon$. 
Recall that $\lambda$ satisfies the Diophantine condition \eqref{eq:diophantine} when $k=1$, i.e. $\mu=|\lambda-1|\geq c$. As a result, 
$\displaystyle \bar{\rho_2}^{-1} < \mu^{-1} \bar{\rho}^{-1} \leq \frac{1}{c}\bar{\rho}^{-1}\lesssim \bar{\rho}^{-1}$ and we also get
\[
  \|\Delta q\|_{\gamma_0 \bar{\rho}_2, 3}  
  \le \frac{\|\Delta q\|_{\bar\rho_2}}{(1 - \gamma_0)^3 \bar{\rho}_1^3}
  \lesssim \frac{\epsilon}{(1 - \gamma_0)^3 \bar{\rho}_1^3}.
\]
Then
\[
  \|\cE(\qn, \varphi)\|_{\gamma_0 \bar\rho_1}
  \le \|\cE(q, \varphi)\|_\rho + \|\cE(\Delta q, \varphi)\|_{\gamma_0 \bar\rho_1}
  \lesssim \epsilon + \|\Delta q\|_{\gamma_0 \bar\rho_2, 1} 
  \lesssim \frac{\epsilon}{(1 - \gamma_0)^3 \bar{\rho}_1^3}
\]
To abreviate notations let us set
\begin{equation}  \label{eq:esp1}
  \epsilon_1 = \frac{\epsilon}{(1 - \gamma_0)^3 \bar{\rho}_1^3}.
\end{equation}

We now show that $[S_{\qn}] - 1$ is of order $\epsilon^2$ when measured on a smaller radius. First of all
\[
  \|[S_{\qn}] - 1\|_{\gamma_0 \bar{\rho}_1}
  \le \|[S_q] - 1\|_{\gamma_0 \bar{\rho}_1} + \|[S_{\Delta q}]\|_{\gamma_0 \bar\rho_1}
  \lesssim \rho\epsilon + \|\Delta q\|_{\gamma_0 \bar\rho_2}
  \lesssim \epsilon. 
\]
Secondly,
\[
  \|[S_{\qn}] - 1\|_{\gamma_0 \gamma_1 \bar\rho_1}
   = \left \| [S_{\qn}] - \Lambda_{2M} [S_{\qn}] \right\|_{\gamma_0 \gamma_1 \bar\rho_1}
  \le \gamma_1^M \|[S_{\qn}] - 1\|_{\gamma_0 \bar\rho_1} \lesssim \epsilon^2,
\]
where we used condition (b).

We are now ready to estimate $\Delta \varphi$. By Lemma~\ref{lem:basic-est}, we have $\|h\|_{\rho_1}, \|1/h\|_{\rho_1} \lesssim 1$ as long as \eqref{eq:small-q} holds with $\rho = \bar{\rho}_1$. To ensure this condition we need to impose
\[
  \epsilon < (1 - \gamma_0)^3 \gamma_0^3 \gamma_2^3 \rho^3, 
\]
so that $\epsilon_1  <  1$, to get  
\[
  \|\qn\|_{\gamma_0 \bar{\rho}_2} \lesssim  \epsilon_1 < 1. 
\]
According to Item 7 of Lemma \ref{lem:basic-est}, \eqref{eq:small-q} holds if 
\[
  C \mu \bar{\rho}_1 (1 + C_1 \bar{\rho}^2) <
  C \mu \bar{\rho}_1 \gamma_0^{-1} < \frac12 |s_0| 
\]
for some constant $C$ dependding only on $C_0, c$ and $\tau$. Noting that $1 + C_1 \bar{\rho}^2 < \gamma_0^{-1}$, this condition is ensured if
\[
  \rho_0 < \frac{\gamma_0}{2C\mu} |s_0|.
\]
This is guaranteed by condition (d).

Denote $\delta = \log(1/\gamma_0)$. From item 6 of Lemma~\ref{lem:basic-est},
\[
  \|\psi\|_{\gamma_0 \bar\rho_1} \lesssim \delta^{-\tau} \epsilon_1, 
\]
and
\[
  \|\Delta \varphi\|_{\gamma_0^2 \bar\rho_1} = \|w\|_{\gamma_0^2 \bar\rho_1}
  \lesssim \delta^{-\tau} \|\psi\|_{\gamma_0 \bar\rho_1} \|1/h\|_{\gamma_0 \bar\rho_1} 
  \lesssim \delta^{-2\tau} \epsilon_1. 
\]

Finally, we will estimte $\cE(\qn, \phin)$. This requires estimating the remainder terms $R_1$ to $R_5$. We have, by item 10-13 in Lemma~\ref{lem:basic-est}
\[
  \|R_1\|_{\gamma_0^2 \gamma_1 \bar\rho_1} 
  \lesssim \frac{1}{\delta \gamma_0^2 \gamma_1 \bar\rho_1} \|[S_{\qn}] - 1\|_{\gamma_0 \gamma_1 \bar\rho_1} 
  \lesssim \frac{\epsilon^2}{\delta \gamma_0^2 \gamma_1 \bar\rho_1}.
\]
\[
  \|\kappa - [\kappa]\|_{\gamma_0^3 \bar\rho_1} 
  \lesssim \frac{\|\cE\|_{\gamma_0 \bar{\rho}_1}}{\delta^\tau (1 - \gamma_0) \bar{\rho}_1}
  = \frac{\epsilon_1}{\delta^\tau (1 - \gamma_0) \bar{\rho}_1},
\]
\[
  \|R_2\|_{\gamma_0^3 \gamma_1 \bar\rho_1} 
  \lesssim \|\psi\|_{\gamma_0 \bar\rho_1} \|\kappa - [\kappa]\|_{\gamma_0^3 \bar\rho_1} 
  \lesssim \frac{\epsilon_1^2}{\delta^{2\tau} (1 - \gamma_0) \bar{\rho}_1}.
\]
\[
  \|R_3\|_{\gamma_0^3  \gamma_1 \bar\rho_1}
  \lesssim \frac{\|[S_{\qn}] - 1\|_{\gamma_0^2 \gamma_1 \bar\rho_1}}{\gamma_0^3 \gamma_1 \delta \bar{\rho}_1}  + \|R_1 + R_2\|_{\gamma_0^3 \bar\rho_1}  
  \lesssim \frac{\epsilon^2_1}{\gamma_0^3 (1 - \gamma_0) \gamma_1 \bar\rho_1 \delta^{2\tau}},
\]
\[
  \begin{aligned}
  \|R_4\|_{\gamma_0^3 \gamma_1 \bar\rho_1} 
  & \lesssim \frac{\|\cE(\qn, \varphi)\|_{\gamma_0 \gamma_1 \bar\rho_1}\|w\|_{\gamma_0 \gamma_1 \bar\rho_1}}{(1 - \gamma_0) \gamma_0^2 \gamma_1 \bar{\rho}_1 }  + \|R_3\|_{\gamma_0^2 \gamma_1 \bar\rho_1} \\
  & \lesssim \frac{\epsilon_1^2}{(1 - \gamma_0)\gamma_0^2 \delta^{2\tau}\gamma_1 \bar{\rho}_1}
  + \frac{\epsilon^2_1}{\gamma_0^3 (1 - \gamma_0) \gamma_1 \bar\rho_1 \delta^{2\tau}}
  \lesssim \frac{\epsilon^2_1}{\gamma_0^3 (1 - \gamma_0) \gamma_1 \bar\rho_1 \delta^{2\tau}},
  \end{aligned}
\]
\[
  \|R_5\|_{\gamma_0^3 \gamma_1 \bar\rho_1} 
  \lesssim \|\qn\|_{\gamma_0^3 \gamma_1 \bar\rho_2, 3}\|\Delta \varphi\|_{\gamma_0^3 \gamma_1 \bar\rho_1}^2 
  \lesssim \frac{\epsilon_1^2}{\delta^{4\tau}}.
\]
We now have
\[
  \begin{aligned}
  \|\cE(\qn, \phin)\|_{\gamma_0^3 \gamma_1 \bar{\rho}_1}
  & \le \|R_4\|_{\gamma_0^3 \gamma_1 \bar{\rho}_1} 
      + \|R_5\|_{\gamma_0^3 \gamma_1 \bar{\rho}_1}
  \lesssim \frac{\epsilon^2_1}{\gamma_0^3 (1 - \gamma_0) \gamma_1 \bar\rho_1 \delta^{4\tau}} \\
  & \le \frac{\epsilon^2}{(1 - \gamma_0)^6 \gamma_0^3 \gamma_1 \bar{\rho}_1^7 \delta^{4\tau}}
  = \frac{\epsilon^2}{(1 - \gamma_0)^6 \gamma_0^3 \gamma_1 \gamma_2^7 (\log(1/\gamma))^{4\tau} \rho^7}\\
  &\lesssim C(\gamma_0,\gamma_1,\gamma_2,\tau,c) \frac{\epsilon^2}{\rho^7}. 
  \end{aligned}
\]
We note that all the constants are explicit in $\gamma_0, \gamma_1, \gamma_2, \tau$ except for an additional unspecified constants that depends only on $C_0$.
\end{proof}

Let $q^\infty, \varphi^\infty$ be the formal solution to $\cE(q, \varphi) = 0$. This follows from Treschev's theorem, and we also provided an alternative proof in Corollary~\ref{cor:treschev-alt}. Since $q^\infty, \varphi^\infty$ is a formal solution, we can truncate them to polynomials, denoted
\[
  \qu{0}, \phiu{0}, 
\]
so that 
\begin{equation}  \label{eq:truncate}
  \cE(\qu{0}, \phiu{0}) = O_{15}.	
\end{equation}
We will use this pair as the initial value of our KAM iteration. 

\begin{proposition}[KAM induction] \label{prop:KAM}
  Let $\qu{0}, \phiu{0}$ be the polynomials given in \eqref{eq:truncate}. There exist $C_6(\lambda) > 1$, $\rhou{0} > 0$ such that if we choose $\gamma_0 = \gamma_1 = \gamma_2$ satisfying $\bar{\gamma} = \gamma_0^5 < (2/3)^{5/4}$, and
  \[
	\epsu{0} = C_6 \left( \rhou{0}\right)^{15}, \quad
	M_0 = \left\lfloor \frac{\log(1/\rhou{0})}{\log(1/\gamma_1)}\right \rfloor + 1, \quad
	N_0 = 7, 
  \]
  and for $n \ge 0$, 
  \[
	M_{n + 1} = \floor{3M_n/2} + 1, \quad 
	N_{n + 1} = \min\{2N_n, M_n\}, \quad
	\rhou{n+1} = \bar{\gamma} \rhou{n}, \quad 
	\epsu{n+1} = \left( \epsu{n}\right)^{\frac32}, 
  \]
	\[
    \qu{n+1} = \qu{n} + \Delta \qu{n}, \quad \phiu{n+1} = \phiu{n} + \Delta \phiu{n}. 
  \]
  Then for every $n \ge 0$, the conditions of Proposition \ref{prop:iterative} are satisfied for every $n$ and
  \[
	\cE(\qu{n}, \phiu{n}) < \epsu{n}.
  \]
  Moreover, we have
 \begin{equation}  \label{eq:Delta-order}
	\Delta \qu{n} = O_{2N_{n} + 2}, \quad \Delta \phiu{n} = O_{2N_n + 1}, \quad n \ge 0.
 \end{equation} 
\end{proposition}

\begin{proof}
  Let $\rho_0$ be as in Proposition \ref{prop:iterative}.
  Since $\qu{0}, \phiu{0}$ are \emph{explicit} polynomials, $\cE(\qu{0}, \phiu{0}) = O_{15}$ is an explict polynomial. There exists a constant $C_6 > 0$ such that
  \[
	\|\cE(\qu{0}, \phiu{0})\|_{\rho} \le C_6 \rho^{15}
  \]
  for every $\rho \in (0, \rho_0)$.

  We will show inductively that, there existss $C_0 > 0$ and $\rho_{[0]} \in (0, \rho_0)$ such that for all $n \ge 0$,
  \begin{enumerate}[(1)]
   \item
	 \[
	   \|\qu{n}\|_{\mu \rhou{n}(1 + C_1(\rhou{n})^2), 3} \le (1 - 2^{n+1}) C_0, \quad
	   \|\phiu{n}\|_{\rhou{n}, 3} \le (1 - 2^{n+1})C_0. 
	 \]
   \item $\|\cE(\qu{n}, \phiu{n})\|_{\rhou{n}} \le \epsu{n}$. 
   \item Conditions (a) - (e) in Proposition~\ref{prop:iterative} are satisfied with $M = M_n$, $\rho = \rhou{n}$, $\epsilon = \epsu{n}$.
 \end{enumerate} 
 Set
 \[
   C_0 = 2\max\{ \|\qu{0}\|_{\mu \rho_0( 1+ C_1), 3}, \|\phiu{0}\|_{\rho_0, 3}\}.
 \]
 We will set $\rhou{0}$ sufficiently small, depending on conditions that depends only on uniform constants. The estimate
\[
  \|\cE(\qu{0}, \phiu{0})\|_{\rhou{0}} \le \epsu{0}
\]
holds by definition. We now check that conditions (a) - (e) of Proposition \ref{prop:iterative} holds for $n = 0$. For condition (a), we have
\[
  C_4 M_0^{\frac{5}{2}} \left( \rhou{0}\right)^2 
  \lesssim (\log(1/\rhou{0})) \left( \rhou{0}\right)^2  \to 0 
\]
as $\rhou{0} \to 0$, therefore (a) can be satisfied by choosing $\rhou{0}$ small enough. For (b), note that by definition,
\[
  \gamma_1^{M_0} = \left( \rhou{0}\right)^{15} < \epsu{0}.
\]
For condition (c), we have
\[
  (1 + C_1 (\rhou{0})^2)^{-M_0} = e^{- M_0 \log(1 + C ( \rhou{0})^2)}
  \ge e^{- \frac{\log(1/\rhou{0})}{\log(1/\gamma_1)} \log(1 + C (\rhou{0})^2)}
  \to 1 
\]
as $\rhou{0} \to 0$, therefore the condition can be satisfied by taking small enough $\rhou{0}$. Condition (d) can be easily satisfed by choosing $\rhou{0}$ small. Finally, condition (e) is implied by 
\[
  \epsu{0} = C_6 (\rhou{0})^{15} < (1 - \gamma_0)^3 \gamma_2^2 (\rhou{0})^3 
\]
which is clearly true for $\rhou{0}$ small enough.
We now check conditions (a) - (e) for $n \ge 1$.
Indeed, if (a) is satisfied for $M_n$, $\rhou{n}$, it is satisfied for $M_{n+1}$, $\rhou{n+1}$ as long as
  \[
	\bar{\gamma}^2\left( \frac32\right)^{\frac{5}{2}} < 1, \text{ or }
	\bar{\gamma}^2 < \left( \frac23\right)^{\frac{5}{2}}.
  \]
  For (b), note that
  \[
	\gamma_1^{2M_{n+1}} < \gamma_1^{2\frac{3M_n}{2} } <  (\epsu{n})^{\frac32} = \epsu{n+1}, 
  \]
  therefore (b) is satisfied for all $n$ by induction. To verify (c), we will show that there exists $\rho_2 > 0$ such that for all $\rho < \rho_2$ and $M > 0$, 
  \[
	(1 + C_1(\bar{\gamma}\rho)^2)^{-3M/2} > (1 + C_1 \rho^2)^{-M}.
  \]
  This means if $\rhou{0} < \rho_2$, then (c) is satisfied for all $n \ge 0$ by induction. Indeed, taking $\log$ to both sides and canceling $M$, it suffices to show
  \[
    \frac{\frac32\log(1 + C_1\bar{\gamma}^2 \rho^2)}{\log(1 + C_1 \rho^2)} < 1
  \]
  for $\rho$ small enough. Noting that the limit of the left hand side as $\rho \to 0$ is $\frac32 \bar{\gamma}^2$, the claim holds as long as $\bar{\gamma}^2 < \frac23$. Condition (d) is satisfied for all $n$ since $\rhou{n}$ is decreasing. To check condition (e) for $n \ge 1$, we claim that
  \[
	\epsu{n} \left( \rhou{n}\right)^{-3}
  \]
  is decreasing. Since
  \[
	\epsu{n+1} \left( \rhou{n+1}\right)^{-3}
	= \epsu{n} \left( \rhou{n}\right)^{-3} \cdot (\epsu{n})^{\frac12} \bar\gamma^{-1}, 
  \]
  the claim holds if $(\epsu{n})^{\frac12} \bar\gamma^{-1} < 1$, which is ensured by choosing $\rhou{0}$ small.

  Finally, we verify the inductive assumption (1) and (2). Suppose they are satisfied for step $n$, we apply Proposition~\ref{prop:iterative} to get 
  \[
	\|\cE(\qu{n+1}, \phiu{n+1})\|_{\rhou{n}} \le C_\gamma \frac{(\epsu{n})^2}{(\rhou{n})^7}.
  \]
  To verify (2), it suffices to show
  \[
     C_\gamma \frac{(\epsu{n})^{\frac12}}{(\rhou{n})^7} < 1.
  \]
  At $n = 0$, this is possible since $\epsu{0} \lesssim (\rhou{0})^{15}$, therefore $\frac{(\epsu{0})^{\frac12}}{(\rhou{0})^7} \lesssim (\rhou{0})^{\frac12} < 1$ if $\rhou{0}$ is small enough. For $n \ge 1$, this hold by induction, since
  \[
	\frac{(\epsu{n+1})^{\frac12}}{(\rhou{n+1})^7} 
	= \frac{(\epsu{n})^{\frac12}}{(\rhou{n})^7} \cdot \frac{(\epsu{n})^{\frac14}}{\bar{\gamma}^7} 
	< \frac{(\epsu{n})^{\frac12}}{(\rhou{n})^7} 
  \]
  as long as $\epsu{n} < \bar{\gamma}^{28}$. This is ensured by choosing $\rhou{0}$ small at step $0$ and by induction every step afterwards. Assumption (2) is verified.
  
  We now come to assumption (1). By Proposition~\ref{prop:iterative}, 
  \[
	\|\Delta \qu{n}\|_{\mu \rhou{n}(1 + C_1(\rhou{n})^2), 3} 
	< C_\gamma \frac{\epsu{n}}{(\rhou{n})^3},
  \]
  \[
	\|\Delta \phiu{n}\|_{\rhou{n}, 3} 
	\le C_\gamma \frac{\epsu{n}}{(\rhou{n})^3}.
  \]
  We claim that  the right hand side of both inequalities are bounded by $2^{-(n+1)} C_0$. This is the case at $n = 0$ by choosing $\rhou{0}$, hence $\epsu{0}$ small. Moreover, since we can always ensure $(\epsu{n})^{\frac12} \bar\gamma^{-1} < \frac12$, we have
  \[
	\epsu{n+1} \left( \rhou{n+1}\right)^{-3}
	= \epsu{n} \left( \rhou{n}\right)^{-3} \cdot (\epsu{n})^{\frac12} \bar\gamma^{-1}
	< \frac12 \epsu{n} \left( \rhou{n}\right)^{-3} 
  \]
  and therefore the same claim hold by induction. 

  We have completed the inductive step, therefore the inductive claims (1) - (3) hold for all $n$.
  
  Finally, let's prove \eqref{eq:Delta-order}. Since $\min\{M_n, 2N_n\} = 2N_n$, we apply Lemma \ref{lem:order-inc} to get 
  \[
	\cE(\qu{n}, \phiu{n}) = O_{2N_n + 1}
  \]
  for all $n$. \eqref{eq:Delta-order} follows.
\end{proof}

The next statement implies our main theorem.
\begin{corollary}
  The formal power series $\qu{n}$	and $\phiu{n}$ from Proposition \ref{prop:KAM} converge term-wise to $q^\infty = \sum_{k = 0}^\infty a_{2k} t^{2k}$ and $\varphi^\infty = \sum_{j, k} b_{j, k} z^j \bar{z}^k$. In the formal sense, we have
  \[
	\cE(q^\infty, \varphi^\infty) = 0. 
  \]
  Moreover, for every $\displaystyle \alpha > \frac{5}{4}$ there exists $c_1, c_2 > 0$ such that 
  \[
	|q^\infty_{k}| \le e^{ c_1 k +  \alpha k \log k}, \quad 
	|\varphi^\infty_{j, k}| \le e^{c_2 (j + k) + \alpha (j + k) \log (j + k)}. 
  \]
  This implies the series $q^\infty$ and $\varphi^\infty$ are of Gevrey order $1 + \alpha$. 
\end{corollary}
\begin{proof}
  Because of \eqref{eq:Delta-order}, both series $\qu{n}$ and $\phiu{n}$ stabilizes and converges term-wise to a limit series. Let us first show that the sequence $M_n$ and $N_{n+1}$ in \eqref{eq:Delta-order} eventually coincide. Indeed, let $n_0 = \inf \{ n \st M_n \le 2N_n \}$, then for all $n \le n_0$, we have
  \[
	N_n = 2 N_{n-1}, \quad M_n = \floor{3M_{n-1}/2} + 1. 
  \]
  It follows that $n_0 < \infty$ since the contrary will lead to a contradiction. In the sequel, we will only deal with the sequence $M_n$.  

   For $k > M_0$, let $n \in \N$ be the unique integer such that 
  \[
	M_{n} < k \le M_{n+1}. 
  \]
  Using the defnition of $M_n$, it's easy to see that there exists $c_1 > 1$ such that
  \[
	\frac{1}{c_1} \left( \frac32\right)^n < M_n < \left( \frac32\right)^n, 
  \]
  therefore
  \[
	\frac{1}{c_1} \left( \frac32\right)^n < k \le \left( \frac32\right)^{n+1}. 
  \]
  In particular, we have $n < \frac{ \log c_1 + \log k}{\log(3/2)}$.

  It follows from Proposition~\ref{prop:iterative} that 
  \[
	\|\qu{n}\|_{\rhou{n}} \le C_0
  \]
  for all $n \ge 0$. Due to \eqref{eq:Delta-order}, we have
  \[
	  |q^\infty_k|  = |(\qu{n})_k| \le \|\qu{n}\|_{\rhou{n}} \left( \rhou{n}\right)^{-k} \le C_0 \left( \rhou{n}\right)^{-k}.
  \]
  Note that 
  \[
	\log \epsu{n} = (3/2) \log \epsu{n-1} = (3/2)^{n} \log \epsu{0} \le c_1 k\log \epsu{0},  
  \]
  and 
  \[
	(\rhou{n})^{-k} = (\rhou{0})^{-k} \bar{\gamma}^{-nk}
	=  (\rhou{0})^{-k} e^{\log(1/\bar{\gamma}) k\frac{\log c_1 + \log k}{\log(3/2)}}
	\le e^{k \log c_2} e^{\alpha k \log k}, 
  \]
  for some constant $c_2 (\rho[0],\bar{\gamma},c_1) > 0$ and
  \[
	\alpha = \frac{\log(1/\bar{\gamma})}{\log(3/2)}.
  \]
  Since we can choose any $\bar{\gamma}$ that satisfies $\bar{\gamma}^2 < (2/3)^{\frac{5}{2}}$, a simple calculation shows we can pick any $\alpha > \frac{5}{4}$. The estimate for $\qu{\infty}$ follows.
The same calculation applies to $\phiu{\infty}$, with the same exponent $\alpha$, but with possibly different constants. 
\end{proof}

\appendix

\section{Solving for $\Delta q$}
\label{sec:delta-q}

In this section we prove Proposition~\ref{prop:bound-Delta-q}. We will be essentially estimating a weighted norm of a lower triangular matrix. Given power series $q$,  $\varphi$, recall that $\xi = (\varphi^- - \varphi)/2$, $\eta = (\varphi^- + \varphi)/2$, $p(t) = \cos(t)$. Let us write, for all $k \ge 2$,
\[
  [\zeta^{2k} p(\zeta)] = \sum_{j = k}^\infty P_{j, k} (z \bar{z})^j, \quad 
  [S_q] = \sum_{j = 2}^\infty [S_q]_{2j} (z \bar{z})^j. 
\]

\begin{lemma}\label{lem:Delta-q-matrix}
Suppose for some $C_0 > 0$ and $\rho_0 > 0$, 
\[
  \|q\|_{\rho_0, 3}, \|\varphi\|_{\rho_0, 3} < C_0. 
\]
Let $C_1 =2C_0$, Then 
\[
 \begin{aligned}
   |P_{jk}| & \ge \frac{1}{\sqrt{2\pi}} \frac{\mu^{2j}}{\sqrt{j}}, 	& \text{if } j = k, \\
   |P_{jk}| &  \le C_1 k \mu^{2k} (1 + C_1 \rho^2)^{2k} \rho^{2k - 2j + 2},  &\text{if }  j > k.  
 \end{aligned} 
\]
\end{lemma}
For all the proofs in this section, $f \lesssim g$ stands for $f \le C g$ for a constant $C > 0$ depending only on $C_0$ and the constant $c$ in the Diophantine condition \eqref{eq:diophantine}. 
\begin{proof}
  We have shown in the proof of Proposition \ref{lem:S-zero-mean} that 
  \[
	P_{j, j} = 2^{-2j} \binom{2j}{j} (\lambda^{-1} + 1)^j (\lambda + 1)^j. 
  \]
  Note that $|\lambda^{-1} + 1| = |\lambda + 1| = \mu$. Applying  Sterling's formula $\displaystyle n! = \sqrt{2\pi n} (\frac{n}{e})^n e^{\frac{\theta}{12n}}, 0<\theta<1$ to the term $\binom{2j}{j}=\frac{(2j)!}{(j!)^2}$, one sees that
  \[
	|P_{j, j}| \ge \frac{1}{\sqrt{2\pi}} \frac{\mu^{2j}}{\sqrt{j}},
  \]

  For the upper bound, we apply Lemma \ref{lem:basic-est} to get $\|\xi\|_\rho \le \mu \rho(1 + C_1 \rho^2)$. Let $\xi_0 = \frac12((\varphi^{(0)})^- + \varphi^{(0)})$. Observe that $\xi_0^{2k}$ only has terms of degree $2k$ and does not contribute to the $j, j$ coefficient when $j \ge k + 1$. We then write
  \[
	\xi^{2k} p(\zeta) - \xi_0^{2k} = (\xi^{2k} - \xi_0^{2k}) + \xi^{2k}(p(\zeta) - 1). 
  \]
  The term $\xi^{2k}$ has degree $2k$, hence does not contribute to the $j, j$ coefficient when $j \ge k + 1$. For the first term, according to Lemma \ref{lem:norm-properties} item 1, item 2 and item 5, we have that 
  \[
	\begin{aligned}
\| (\xi - \xi_0) \sum_{l = 0}^{2k - 1} \xi_0^l \xi^{2k - l - 1}\|_\rho 
& \lesssim \rho^3 \|(\xi - \xi_0) \|_{\rho} (2k)\| \sup_{0\leq l\leq 2k-1} \|\xi_0^{l}\xi^{2k-1-l}\|_{\rho}\\
& \lesssim 2k\rho^3\|(\xi - \xi_0) \|_{\rho}\sup_{0\leq l\leq 2k-1}(\mu \rho)^l(1-c\rho^2)^{2k-1-l}\\
& \lesssim  (k) \rho^3\mu^{2k - 1}\rho^{2k-1}(1 + C_1 \rho^2)^{2k - 1} \\
& \le    k \mu^{2k - 1} \rho^{2k + 2} (1 + C_1 \rho^2)^{2k - 1}.
	\end{aligned}
  \]
  For the second term, we note that $p(\zeta) - 1 = \zeta^2 \frac{p(\zeta) - 1}{\zeta^2}$. Given that $\frac{p(t) - 1}{t^2}$ is an entire function, we have
  \[
	\|p(\zeta) - 1\|_\rho \le \rho^2 \|(p(\zeta) - 1)\|_\rho \lesssim \rho^2.
  \]
  Hence
  \[
	\|\xi^{2k}(p(\zeta) - 1)\|_\rho \lesssim \rho^2 \cdot \mu^{2k} \rho^{2k} (1 + C_1 \rho^2)^{2k} = \mu^{2k} \rho^{2k + 2} (1 + C_1 \rho^2)^{2k}.
  \]
  Combine the estimates, we get 
  \[
	\|\xi^{2k}p(\zeta) - \xi_0^{2k}\|_\rho \lesssim k \mu^{2k} \rho^{2k + 2} (1 + C_1 \rho^2)^{2k} 
  \]
  and due to Lemma \ref{lem:norm-properties} item 1, 
  \[
	|P_{jk}| = |[\xi^{2k}p(\zeta) - \xi_0^{2k}]_{j,j}| \lesssim \rho^{-2j}k \mu^{2k} \rho^{2k + 2} (1 + C_1 \rho^2)^{2k}= k \mu^{2k} \rho^{2k -2j+ 2} (1 + C_1 \rho^2)^{2k}.
  \]
\end{proof}

Denote
\[
  T_M = \begin{bmatrix}
	P_{2, 2} & 0 &  \cdots & 0  \\
	\vdots & \ddots & \ddots & \vdots \\  
	\vdots & \ddots & \ddots & 0 \\  
	P_{M, 2} & \cdots & \cdots & P_{M, M} 
  \end{bmatrix} , 
\]
and 
\[
  \bv = \bmat{\eta_4 \\ \eta_6 \\ \vdots \\ \eta_{2M}}, 
  \quad
  \bw = - \bmat{[S_q]_4 \\ [S_q]_6 \\ \vdots \\ [S_q]_{2M}}, 
\]
then 
\[
  \Lambda_{2M}\left( [S_{q + \Delta q}] - 1\right) = 0
\]
with $\Delta q = \sum_{k = 2}^{M} \eta_{2k} t^{2k}$ if and only if 
\[
  T_M \bv = \bw. 
\]

For $r > 0$, denote 
\[
  \Gamma_r = \diag\{r^4, r^6, \cdots, r^{2M}\}. 
\]

\begin{lemma}
Given $\rho_1, \rho_2 > 0$, let $\|\cdot\|_1$ be the operator norm induced by the standard 1-norm, we have that
\[
  \|\Lambda_{2M}(\Delta q) \|_{\rho_2} \le \| \Gamma_{\rho_2} T_M^{-1} \Gamma_{\rho_1}^{-1}\|_1 \cdot \|\Lambda_{2M}[S_q] - 1\|_{\rho_1}. 
\]
\end{lemma}
\begin{proof}
Note that
\[
  \|\Lambda_{2M}(\Delta q)\|_{\rho_2}
  = \sum_{k = 2}^M |\eta_{2k}| \rho^{2k} 
  = \|\Gamma_{\rho_2} \bv\|_1,
\]
\[
  \|\Lambda_{2M}([S_q] - 1)\|_{\rho_1}
  = \|\Gamma_{\rho_1} \bw \|_1. 
\]
Therefore
\[
  \begin{aligned}
  \|\Lambda_{2M}(\Delta q)\|_{\rho_2}
  & =  \|\Gamma_{\rho_2} \bv\|_1 
  = \|\Gamma_{\rho_2} T_M^{-1} \bw\|_1
  = \|\Gamma_{\rho_2} T_M^{-1} \Gamma_{\rho_1}^{-1} \Gamma_{\rho_1} \bw\|_1 \\
  & \le \| \Gamma_{\rho_2} T_M^{-1} \Gamma_{\rho_1}^{-1}\|_1 
  \|\Gamma_{\rho_1} \bw\|_1 
  = \| \Gamma_{\rho_2} T_M^{-1} \Gamma_{\rho_1}^{-1}\|_1 \cdot \|\Lambda_{2M}[S_q] - 1\|_{\rho_1}. 
  \end{aligned}
\]
\end{proof}

Finally, Proposition~\ref{prop:bound-Delta-q} follows from:
\begin{lemma}
  Under the assumption of Lemma~\ref{lem:Delta-q-matrix} and Proposition~\ref{prop:bound-Delta-q}, there existss $C_2 > 0$ such that for 
  \[
	\rho_1 = \frac{\rho}{(1 + C_1 \rho^2)^M}, \quad 
	  \rho_2 = \mu\rho_1, 
  \]
  we have
  \[
	\| \Gamma_{\rho_2} T_M^{-1} \Gamma_{\rho_1}^{-1}\|_1  \le 2C_2 \sqrt{M}.
  \]
\end{lemma}

\begin{proof}
Denote
\[
  \tilde{T}_M = \Gamma_{\rho_1} T_M \Gamma_{\rho_2} = (\tilde{P}_{jk})_{2 \le j, k \le M}. 
\]
Note that for consistency of notations, we are starting the index at $2$. We have
\[
  \tilde{P}_{j, k} = \rho_1^{2j} P_{j, k} \rho_2^{-2k}, 
\]
hence
\[
  |\tilde{P}_{j, j}| = |\rho_1^{j}P_{j,j}(\mu\rho)^{-2j}| = |\mu^{-2j}P_{j,j}|\geq |\mu^{-2j} \frac{1}{\sqrt{2\pi}}\frac{\mu^{2j}}{\sqrt{j}} |= \frac{(2\pi)^{-\frac{1}{2}}}{\sqrt{j}}. 
\]
For $j - k \ge 1$, $k\leq M$, we have that
\[
\begin{aligned}
  |\tilde{P}_{j, k}| & \lesssim M \rho_1^{2j} \mu^{2k} (1 + C_1 \rho^2)^{2k} \rho^{2k-2j+2}\mu^{-2k} \rho_1^{-2k} \\
					 & \le M (1 + C_1 \rho^2)^{2M}\rho^{2} (1 + C_1 \rho^2)^{-(2j - 2k)M}  \\
					 & \le M \rho^2 (1 + C_1 \rho^2)^{2M(k - j + 1)}.
\end{aligned} 
\]

Write $\tilde{T}_M = \Gamma + N$, where $\Gamma$ is the diagonal part of $\tilde{T}_M$. Then formally, 
\[
  \tilde{T}_M^{-1} = (\Gamma + N)^{-1} 
  = (I + \Gamma^{-1} N)^{-1} \Gamma^{-1} = \sum_{k = 0}^\infty (\Gamma^{-1} N)^k \Gamma^{-1}. 
\]
Since
\[
\begin{aligned}
  \|\Gamma^{-1} N\|_1 &
  \le \max_{k} \sum_{j = k+1}^M |P_{jj}^{-1} P_{jk}| 
  = \max_k \sum_{j = k + 1}^M  M \sqrt{j} \rho^2 (1 + C_1 \rho^2)^{2M(k - j + 1)} \\
					  & \le \max_k \sum_{j = k+1}^MM\sqrt{j}\rho^2  \\
					  &\leq \sum_{j=1}^{M}M\sqrt{M}\rho^2 = M^{\frac{5}{2}}\rho^2 
\end{aligned}
\]
Hence the condition \eqref{eq:M52rho2} implies that $\|\Gamma^{-1} N\|_1 < \frac12$, which in turn leads to
\[
  \|\tilde{T}_M^{-1}\|_1 \lesssim \|\Gamma^{-1}\|_1 \lesssim \sqrt{M}. 
\]
\end{proof}

\bibliographystyle{alpha}
\bibliography{linearizable-billiard}

\end{document}